\documentclass[a4paper,12pt]{amsart}

\usepackage{fullpage}
\usepackage{latexsym,amsfonts,amssymb,amsmath,amsthm, color}

\usepackage[top=1in, bottom=1in, left=1in, right=1in]{geometry}
\usepackage{graphicx}
\usepackage{epstopdf}
\usepackage{verbatim}
\usepackage{enumerate}

\newtheorem{thm}{Theorem}[section]

\newtheorem{prop}[thm]{Proposition}
\newtheorem{lem}[thm]{Lemma}
\newtheorem{lemma}[thm]{Lemma}
\theoremstyle{remark}
\newtheorem{rem}[thm]{Remark}
\newcommand{\sumstar}{\sideset{}{^*}\sum}
\newcommand{\sumb}{\sideset{}{^\flat}\sum}

\newcommand{\es}[1]{\begin{equation}\begin{split}#1\end{split}\end{equation}}
\newcommand{\est}[1]{\begin{equation*}\begin{split}#1\end{split}\end{equation*}}

\newcommand{\tRe}{\textup{Re }}
\newcommand{\tIm}{\textup{Im }}
\newcommand{\Arg}{\textup{Arg }}
\newcommand{\bfrac}[2]{\left(\frac{#1}{#2}\right)}

\newcommand{\chiq}{\chi \textup{ (mod $q$)}}

\newcommand{\cb}{\overline{\chi}}
\newcommand{\V}{V}
\newcommand{\W}{\mathcal W}
\newcommand{\D}{\mathcal D (\Psi, Q)}

\newcommand{\Smo}{\mathcal{S}_1 (\Psi, Q)}
\newcommand{\Smtwo}{\mathcal{S}_2 (\Psi, Q)}

\newcommand{\Mg}{\mathcal{MG}(\Psi, Q)}
\newcommand{\Eg}{\mathcal{EG}(\Psi, Q)}
\newcommand{\G}{\mathcal G (\Psi, Q)}

\newcommand{\Hc}{\mathcal H}

\newcommand{\F}{\mathcal F}
\newcommand{\R}{{\mathrm{Re}}}

\newcommand{\h}{\frac{1}{2}} 

\newcommand{\lon}{Q_0}
\newcommand{\lotwo}{Q_0^2}
\newcommand{\phib}{\phi^\flat(q)}
\newcommand{\dt}{\widetilde{\Delta}}
\newcommand{\Wt}{\widetilde{\mathcal{W}}}
\newcommand{\psit}{\widetilde{\Psi}}

\begin{document}
\title{The eighth moment of Dirichlet $L$-functions {II} }
\author{Vorrapan Chandee}
\address{Mathematics Department, Kansas State University \\ 138 Cardwell Hall Manhattan, KS 66506 USA}

\email{chandee@ksu.edu}

\author{Xiannan Li}
\address{Mathematics Department, Kansas State University \\ 138 Cardwell Hall Manhattan, KS 66506 USA}
\email{xiannan@ksu.edu}

\author{Kaisa Matom\"aki}
\address{Department of Mathematics and Statistics \\ University of Turku \\ 20014 Turku \\ Finland}
\email{ksmato@utu.fi}

\author{Maksym Radziwi{\l\l}}
\address{UT Austin, Department of Mathematics, 2515 Speedway, PMA 8.100, Austin, TX 78712}
\email{maksym.radziwill@gmail.com}


\begin{abstract} 
We prove an asymptotic formula for the eighth moment of Dirichlet $L$-functions averaged over primitive characters $\chi$ modulo $q$, over all moduli $q\leq Q$ and with a short average on the critical line. Previously the same result was shown conditionally on the Generalized Riemann Hypothesis by the first two authors.
\end{abstract}

\maketitle
\section{Introduction} 
Moments of $L$-functions have attracted a great deal of attention. Not only do they have numerous applications, but they have also their own intrinsic interest. The first moments studied were naturally those of the Riemann zeta function, which are averages of the form
$$I_k(T) := \int_0^T |\zeta(\tfrac{1}{2} + it)|^{2k}dt.
$$  
An asymptotic formula for $I_k(T)$ was proven for $k = 1$ by Hardy and Littlewood and for $k=2$ by Ingham (see e.g. \cite[Chapter VII]{Ti}). Despite considerable effort, such an asymptotic formula is still not known for any other value of $k$.  

The situation for other $L$-functions is very similar; asymptotics are only available for small values of $k$, and often only when averaged over a suitable family. In case of Dirichlet $L$-functions, Conrey, Iwaniec and Soundararajan \cite{CIS} have proven an asymptotic formula for the sixth moment with an averaging over characters $\chiq$, over all moduli $q \leq Q$ and with a short average on the critical line. The first two authors~\cite{CL8th} have shown a similar asymptotic formula for the eighth moment, conditionally on the GRH. The aim of this paper is to provide an unconditional proof of this eighth moment result. 

Before stating our result, let us introduce some notation. Let $\chiq$ be a primitive even\footnote{The restriction to even characters is only for convenience so that the completed $L$-function has the same shape for all characters involved --- odd characters could be treated in exactly the same way.} Dirichlet character, and let (for $\tRe s > 1$),
$$ L(s, \chi) := \sum_{n=1}^\infty \frac{\chi(n)}{n^s} = \prod_p \left( 1 - \frac{\chi(p)}{p^s}\right)^{-1} $$ 
be the Dirichlet $L$-function associated to it.  Then the completed $L$-function 
\[
\Lambda\big( \tfrac{1}{2} + s, \chi \big) := \left(\frac{q}{\pi}\right)^{s/2} \Gamma \left(\tfrac{1}{4} + \tfrac{s}{2}\right) L\big( \tfrac{1}{2} + s, \chi\big) 
\]
satisfies the functional equation
\begin{equation}
\label{eq:FE}
\Lambda\big( \tfrac{1}{2} + s, \chi \big) = \varepsilon_\chi \Lambda\big( \tfrac{1}{2} - s, \overline{\chi}\big),
\end{equation}
where $|\varepsilon_\chi| = 1$.  

Let $\sumb_{\chiq}$ denote a sum over primitive even Dirichlet characters with modulus $q$, and $\phib$ denote the number of primitive even Dirichlet characters with modulus $q$. From \cite{CFKRS}, one may derive the conjecture that as $q \rightarrow \infty$ with $q \neq 2$ (mod 4),
$$ \frac{1}{\phib} \sumb_{\chiq} \big|L\big(\tfrac{1}{2}, \chi \big)\big|^8 \sim 24024 \ a_4 \prod_{p | q} \frac{\left( 1 - \frac{1}{p}\right)^7}{\left( 1 + \frac{9}{p} + \frac{9}{p^2} + \frac{1}{p^3}\right)} \frac{(\log q)^{16}}{16!},$$
where 
\[
a_4 := \prod_{p} \left( 1 - \frac{1}{p}\right)^9 \left(1 + \frac{9}{p} + \frac{9}{p^2} + \frac{1}{p^3} \right).
\]

Toward this conjecture we prove the following asymptotic formula when there is an additional $q$-average as well as a $t$-average which is very short thanks to the rapid decay of the $\Gamma$-function. Exactly the same theorem was shown in~\cite{CL8th} conditionally on the Generalized Riemann Hypothesis.
\begin{thm}\label{thm:eightmoment} Let $\varepsilon > 0$ and let $\Psi$ be a smooth function compactly supported in $[1, 2]$. Then
\begin{equation}
\label{eq:ThmClaim}
\begin{split}
& \sum_{q } \Psi\left(\frac{q}{Q}\right) \sumb_{\chiq} \int_{-\infty}^{\infty} \left| \Lambda\big( \tfrac{1}{2} + it, \chi \big)\right|^8 \> dt \\
&= 24024 a_4\sum_{q}  \Psi\left(\frac{q}{Q}\right) \prod_{p|q} \frac{\left( 1 - \frac{1}{p}\right)^7}{\left( 1 + \frac{9}{p} + \frac{9}{p^2} + \frac{1}{p^3}\right)} \phib \frac{(\log q)^{16}}{16!} \int_{-\infty}^{\infty} \left| \Gamma \left( \frac{1}{4} + \frac{it}{2}\right)\right|^8 \> dt \\
& \hskip 1in + O_\varepsilon(Q^2(\log Q)^{15+ \varepsilon}).
\end{split}
\end{equation}
\end{thm}

\begin{rem}
Note that the main term in the theorem is of the order $Q^2 (\log Q)^{16}$ and one obtains as a corollary that
\begin{align*}
 \sum_{q \leq Q} \ \sumb_{\chiq} &\int_{-\infty}^{\infty} \left| \Lambda\big( \tfrac{1}{2} + it, \chi \big)\right|^8  \> dt \\
&\sim 24024 a_4\sum_{q \leq Q} \prod_{p|q} \frac{\left( 1 - \frac{1}{p}\right)^7}{\left( 1 + \frac{9}{p} + \frac{9}{p^2} + \frac{1}{p^3}\right)} \phib \frac{(\log q)^{16}}{16!} \int_{-\infty}^{\infty} \left| \Gamma \left( \frac{1}{4} + \frac{it}{2}\right)\right|^8 \> dt. 
\end{align*}
\end{rem}

\begin{rem}
In \cite{CLMR6th}, the present authors will remove the $t$-average from the sixth moment in the work of Conrey, Iwaniec, and Soundararajan~\cite{CIS}. In particular, we will show that

	\begin{align*}
	\sum_{q\leq Q} \;\; &\sumb_{\chiq}  \left| L\big( \tfrac{1}{2} , \chi \big)\right|^6  \sim 42 a_3\sum_{q\leq Q}  \prod_{p|q} \frac{\left( 1 - \frac{1}{p}\right)^5}{\left( 1 + \frac{4}{p} + \frac{1}{p^2}\right)} \phib \frac{(\log q)^{9}}{9!} ,
	\end{align*}where
$$ a_3 = \prod_p \left(1-\frac{1}{p^4}\right)\left(1+\frac{4}{p} + \frac{1}{p^2}\right).$$
However, it remains challenging to remove the $t$-average for the eighth moment.

\end{rem}

\section{A sketch of the proof}
In this section we provide a sketch of the proof where we ignore various technicalities such as complicated smooth weights, the inclusion-exclusion within the orthogonality over primitive characters and a number of coprimality conditions and common divisors. 

Roughly speaking, after applying the approximate functional equation, we need to understand sums of the form
\begin{equation}
\label{eq:SketchMainSum}
Q\sum_{q} \Psi\bfrac{q}{Q} \sum_{m \leq Q^2} \sum_{\substack{ n \leq Q^2 \\ m\equiv n \pmod{q}}} \frac{\tau_4(m)\tau_4(n)}{\sqrt{mn}},
\end{equation}
where $\tau_4(n) = \sum_{n = n_1n_2n_2n_4} 1.$ Here we were able to make the restriction $m, n \leq Q^2$ instead of just $m n \leq Q^4$ thanks to the $t$-average in the theorem. 

The diagonal contribution $m=n$ in~\eqref{eq:SketchMainSum} is fairly easy to understand, and in this sketch we shall concentrate on the non-diagonal contribution. Let $\varepsilon_0 > 0$ be small but fixed. Following~\cite{CL8th} the sums over $m$ and $n$ in~\eqref{eq:SketchMainSum} can be truncated to 
\[
m, n \leq N:=  \frac{Q^2}{\exp((\log Q)^{\varepsilon_0})}
\]
using the multiplicative large sieve. Still following~\cite{CL8th} (and~\cite{CIS}) we apply in the most critical range the complementary divisor trick. That is we write $m - n = hq$ in~\eqref{eq:SketchMainSum} and replace the congruence condition modulo $q$ with a congruence condition modulo $h$.  Note that $h\ll N/q$ is smaller than $Q$, so we have a reduction in the arithmetic conductor.

After switching to the complementary divisor, we express the congruence condition modulo $h$ using characters modulo $h$, so that roughly we want to study
\begin{equation}
\label{eq:SketchCompl}
Q\sum_{h \leq 2N/Q} \frac{1}{\phi(h)}\sum_{\chi \textup{(mod $h$)}}\sum_{m, n} \frac{\tau_4(m)\tau_4(n)\chi(m) \overline{\chi(n)}}{\sqrt{mn}}\Psi\bfrac{|m-n|}{hQ} \Psi_1\left(\frac mN\right) \Psi_1\left(\frac nN\right),
\end{equation}
where $\Psi_1$ is smooth and supported on $[0, 1]$.

The principal characters give a main term contribution. In this sketch we concentrate on the non-principal characters. The smooth factor 
\[
h(m, n) := \Psi\bfrac{|m-n|}{hQ} \Psi_1\left(\frac mN\right) \Psi_1\left(\frac nN\right)
\]
restricts $m$ and $n$ to being within distance $2hQ$ from each other.  Morally, the short interval type condition $|m-n| \ll 2hQ$ introduces an archimedean conductor of size $T = \frac{N}{hQ}$.  The hybrid conductor is then $hT \asymp \frac{N}{Q}$, and this is still smaller than the original conductor $Q$.  It is important that the sums of length $N \asymp \frac{Q^2}{\exp((\log Q)^{\varepsilon_0})}$ are long compared to the hybrid conductor.  In particular, applying Fourier analysis to such a sum with $\tau_4(n)$ as the coefficient produces dual sums of length $\frac{(N/Q)^4}{N} \asymp \frac{N^3}{Q^4} \ll \frac{Q^2}{\exp(3(\log Q)^{\varepsilon_0})}$, and this is shorter than the length of the original sum $N = \frac{Q^2}{\exp((\log Q)^{\varepsilon_0})}$.  Actually for technical convenience we use the approximate functional equation rather than the functional equation, and this gives us sums of length $\bfrac{N}{Q}^2 \asymp \frac{Q^2}{\exp(2(\log Q)^{\varepsilon_0})}$, which still suffices.  This is the main motivation for the more technical arguments that follow.

To be more precise, we will reduce our problem to that of bounding a mean square of a corresponding Dirichlet series, in the spirit of arithmetic problems on almost all short intervals (see e.g.~\cite[Lemma 9.3]{HarmanBook}).  
Indeed, one can show that the Mellin transform
\begin{equation}
\label{eq:hMellin}
\widetilde{h}(s_1, s_2) = \int_0^\infty \int_0^\infty h(x, y) x^{s_1} y^{s_2} \frac{dy}{y} \frac{dx}{x} 
\end{equation}
converges for $\tRe s_i > 0$ and satisfies, for $\tRe s_i \in (0, 100)$,
\begin{equation}
\label{eq:hMellinBound}
\widetilde{h}(s_1, s_2) \ll \frac{1}{\tRe s_1 \cdot \tRe s_2} \left(\frac{N}{hQ}\right)^{k-1} \frac{1}{\max\{|s_1|, |s_2|\}^{k} |s_1 + s_2|^{l}} N^{\tRe s_1 + \tRe s_2} 
\end{equation}
for any integers $k \geq 1$ and $l \geq 0$. In showing this, one can assume $|s_1| \geq |s_2|$, and in this case~\eqref{eq:hMellinBound} follows by applying in~\eqref{eq:hMellin} first partial integration $k$ times with respect to $x$, then substituting $y = xz$ and finally applying partial integration $l$ times with respect to $x$. A similar argument with our more complicated weight function can be found from the proof of Lemma~\ref{lem:MellinXY}.

Now by the Mellin inversion, the non-principal characters contribute to~\eqref{eq:SketchCompl}
\[
Q \sum_{h \leq 2N/Q} \frac{1}{\phi(h)}\sum_{\substack{\chi \textup{(mod $h$)} \\ \chi \neq \chi_0}} \int_{(1/2+\varepsilon)} \int_{(1/2+\varepsilon)} L(1/2+s_1, \chi)^4 L(1/2+s_2, \overline{\chi})^4 \widetilde{h}(s_1, s_2) ds_1 ds_2.
\]
Moving the contours close to $\tRe s_i = 0$, plugging in the bound~\eqref{eq:hMellinBound} for $\widetilde{h}(s_1, s_2)$, and using the inequality $xy \leq |x|^2 + |y|^2$, this can be seen to be essentially
\[
\ll Q \sum_{h \leq 2N/Q} \frac{1}{\phi(h)}\sum_{\chi \textup{(mod $h$)}} \frac{hQ}{N} \int_{-N/(hQ)}^{N/(hQ)} \left|L(1/2+it, \chi)\right|^8 dt.
\]
Hence our task more-or-less reduces to showing that, for any $H \leq N/Q$ and $C \geq 1$, we have
\begin{equation}
\label{eq:sketch8int}
\frac{Q^2}{N} \sum_{h \sim H}\sum_{\chi \textup{(mod $h$)}} \int_{-T}^{T} \left| L(1/2+it, \chi)\right|^8 dt \ll \frac{Q^2}{(\log Q)^C} \quad \text{for $T = N/(HQ)$}.
\end{equation}
Here and later the notation $h \sim H$ in summations means that $H < h \leq 2H$. Since $HT = N/Q$, the approximate functional equation morally allows us to approximate $|L(1/2+it, \chi)|^4$ by $|\sum_{n \ll (N/Q)^2} \chi(n) \tau_4(n) n^{-1/2-it}|$ (see Proposition~\ref{prop:8momentIntandSumoverq} below for a rigorous argument) and hence by the large sieve (Lemma~\ref{lem:largesieve} below) the left hand side of~\eqref{eq:sketch8int} can be shown to be
\[
\ll \frac{Q^2}{N} \left(H^2 T + \left(\frac{N}{Q}\right)^2 \right) (\log N)^{O(1)} \ll \left(QH + N\right) (\log N)^{O(1)}. 
\]
We recall that $H \leq N/Q$ and  $N = \frac{Q^2}{\exp((\log Q)^{\varepsilon_0})},$ so the above is $O(Q^2/(\log Q)^{C})$ for any $C \geq 1$.


The current paper has a lot in common with~\cite{CL8th} and we freely borrow results from there, so the reader may want to have~\cite{CL8th} at hand. However, if the reader is ready to take those results granted or to work them out, the current paper can be read alone.

Throughout the paper, $\varepsilon$ denotes a small positive real
number. Furthermore $\varepsilon_0$ and $\Delta_0$ will be fixed positive constants that are chosen later. 

\section{Large sieve and upper bounds for moments}

Let us first recall the hybrid large sieve (see e.g.~\cite[Theorem 5.1]{RichertNotes}).
\begin{lemma}\label{lem:largesieve} 
Let $T, Q \geq 1$. For any complex coefficients $a_n$ with $\sum_{n = 1}^\infty |a_n| < \infty$, one has
\[
\sum_{q \leq Q} \, \sumstar_{\chi \ ({\rm mod} \ q)} \frac{q}{\phi(q)} \int_{-T}^T \left| \sum_{n=1}^\infty a_n \chi(n) n^{it} \right|^2 \ll \sum_{n=1}^\infty (Q^2 T + n) |a_n|^2.
\]
In particular, for any $N \geq 1$ and any complex coefficients $a_n$, one has
\[
\sum_{q \leq Q} \,\sumstar_{\chi \ ({\rm mod} \ q)} \frac{q}{\phi(q)} \int_{-T}^T \left| \sum_{n \leq N} a_n \chi(n) n^{it} \right|^2 \ll (Q^2 T + N) \sum_{n \leq N} |a_n|^2.
\]
\end{lemma}

The following proposition gives an upper bound for the eight moment of Dirichlet $L$-functions averaged over $\chi, q$ and $t$. It will be used in bounding the off-diagonal terms in Section \ref{sec:errorEg}. 

\begin{prop} \label{prop:8momentIntandSumoverq} For $Q, T \geq 3$, one has
\begin{align*}
& \sum_{q \leq Q} \sum_{\chi \ ({\rm mod} \ q)} \int_{0}^T \left|L\left( \frac{1}{2} + c + it, \chi \right) \right|^{8} dt \ll Q^2 T^2  (\log (QT))^{16},
\end{align*}whenever $0\leq c \leq 1/100$.
\end{prop}
\begin{rem}
In~\cite{CL8th} it was shown that conditionally on GRH, one has
\begin{equation}
\label{eq:CL8thupperGRH}
\sum_{\chi \ ({\rm mod} \ q)} \int_{0}^T \left|L\left( \frac{1}{2} + c + it, \chi \right) \right|^{8} dt \ll q T  (\log (qT))^{16+\varepsilon}.
\end{equation}
The unconditional Proposition~\ref{prop:8momentIntandSumoverq} suffices for us since we will afford to lose a factor $T$ due to the decay in $|s_1+s_2|$ coming from Lemma~\ref{lem:MellinXY} below. On the other hand the fact that we need an additional average over $q$ in Proposition~\ref{prop:8momentIntandSumoverq} compared to~\eqref{eq:CL8thupperGRH} will cause us some minor technical difficulties (see in particular Remark~\ref{rem:aleqA}).
\end{rem}
\begin{proof}[Proof of Proposition~\ref{prop:8momentIntandSumoverq}]
The contribution from $q=1$ is acceptable by known moment bound for the Riemann zeta function (see Lemma~\ref{le:zeta8th} below). Moreover, the part with bounded $t$ can be easily dealt with using the approximate functional equation and the large sieve (see~\cite[Theorem 7.34]{iw-kow}, the proof works just as well with $16$ in place of $17$). 

Hence, by dyadic splitting, it suffices to prove, for $Q \geq 3/2$ and $T \geq 3$,
\begin{equation}
\label{eq:8thupdyadclaim}
\sum_{q \sim Q} \sum_{\chi \ ({\rm mod} \ q)} \int_{T}^{2T} \left|L\left( \frac{1}{2} + c + it, \chi \right) \right|^{8} dt \ll Q^2 T^2  (\log (QT))^{16}.
\end{equation}
As usual, we first reduce to primitive characters; we claim that~\eqref{eq:8thupdyadclaim} follows once we have shown that
\begin{equation}
\label{eq:8thupprim}
\sum_{q \sim Q} \ \sumstar_{\chi \ ({\rm mod} \ q)} \int_{T}^{2T} \left|L\left( \frac{1}{2} + c + it, \chi \right) \right|^{8} dt \ll Q^2 T^2  (\log (QT))^{16}.
\end{equation}
Indeed, the left hand side of~\eqref{eq:8thupdyadclaim} is
\begin{equation}
\label{eq:8MomRoughPrimSplit}
\begin{split}
&\leq \sum_{r \leq Q} \sum_{q_1 \sim \frac{Q}{r}} \  \sumstar_{\chi_1 \ ({\rm mod} \ q_1)} \int_{T}^{2T} \left|L\left( \frac{1}{2} + c+it, \chi_1 \right) \right|^{8}  \prod_{\substack{p|r}} \left(1 + \frac{1}{p^{1/2}}\right)^{8}  \>dt
\end{split}
\end{equation}
Applying~\eqref{eq:8thupprim}, this is at most
\[
Q^2 T^2 (\log (QT))^{16} \sum_{r \leq Q} \frac{1}{r^2} \prod_{\substack{p|r}} \left(1 + \frac{1}{p^{1/2}}\right)^{8}   \ll Q^2 T^2 (\log (QT))^{16}
\]
as claimed.
Hence it suffices to prove~\eqref{eq:8thupprim}. For $\chi$ a primitive character mod $q$, by the approximate functional equation we have morally that
\[
\left|L\left( \frac{1}{2} + it, \chi \right)^4 \right| \ll \left|\sum_{n \ll (TQ)^{2}} \frac{\tau_4(n) \chi(n)}{n^{1/2+it}}\right| + \left|\sum_{n \ll (TQ)^{2}} \frac{\tau_4(n)\overline{\chi}(n)}{n^{1/2-it}}\right|.
\]
If this was true, the claim~\eqref{eq:8thupprim} (for $c = 0$) would follow from the hybrid large sieve (Lemma~\ref{lem:largesieve}). However, there is a technical issue that the Dirichlet polynomials in the approximation depend mildly on $t$ and $q$. To proceed rigorously, we use a method of Ramachandra~\cite{Ramachandra}.

Let us concentrate on the sum in~\eqref{eq:8thupprim} over even characters, the odd characters being handled similarly. By the functional equation~\eqref{eq:FE}, we have
$$
L(\tfrac 12 + s, \chi)^4 = F(\tfrac 12 + s) L(\tfrac 12 - s, \overline{\chi})^4
$$
with
$$
F(\tfrac 12 + s) := \varepsilon_{\chi}^4 \Big ( \frac{\pi}{q} \Big )^{4 s} \frac{\Gamma \Big ( \frac{1}{4} -\frac{s}{2} \Big )^4}{\Gamma \Big ( \frac{1}{4} + \frac{s}{2} \Big )^4},
$$
where $|\varepsilon_{\chi}| = 1$. For further convenience, let us record here Stirling's formula which gives, for $|\Arg(z)| < \pi -\varepsilon$, 
\begin{equation}
\label{eq:Stirling}
\Gamma(z) = \sqrt{\frac{2\pi}{z}} \left(\frac{z}{e}\right)^z \left(1+O(1/z)\right) \ll \frac{|z|^{\tRe z - 1/2}}{e^{\tRe z}} \exp(- \tIm z \cdot \Arg(z)).
\end{equation}
In particular, for $\tRe s \in [-1, 1/3]$ and $|\tIm s| \geq 10$, we have 
\begin{equation}
\label{eq:GammaDivInt}
\frac{\Gamma (\frac{1}{4} -\frac{s}{2})}{\Gamma(\frac{1}{4} + \frac{s}{2})} \asymp \frac{\exp(\tRe(-\frac{1}{4} -\frac{s}{2}) \log |\frac{1}{4} -\frac{s}{2}| - \textrm{Im}(\frac{-s}{2}) \Arg(\frac{1}{4} -\frac{s}{2}))}{\exp(\tRe(-\frac{1}{4} +\frac{s}{2}) \log |\frac{1}{4} + \frac{s}{2}| - \textrm{Im}(\frac{s}{2}) \Arg(\frac{1}{4} +\frac{s}{2}))}
\end{equation}
and
\[
\log \left|\frac{1}{4}\pm \frac{s}{2}\right| = \log(|s|+1) + O(1) \quad \text{and} \quad \Arg\left(\frac{1}{4} \pm \frac{s}{2}\right) = \Arg\left(\pm \frac{s}{2}\right) + O\left(\frac{1}{|s|+1}\right).
\]
Noticing that the left hand side of~\eqref{eq:GammaDivInt} stays bounded when $\tRe s \in [-1, 1/3]$ and $|\tIm s| < 10$, we obtain that, for any $s$ with $\tRe s \in [-1, 1/3]$,
\begin{equation}
\label{eq:Gammadiv}
\frac{\Gamma (\frac{1}{4} -\frac{s}{2})}{\Gamma(\frac{1}{4} + \frac{s}{2})} \ll \exp(-\tRe (s) \log(|s|+1)) \asymp (|s|+1)^{-\tRe s}.
\end{equation}



The starting point in the proof of~\eqref{eq:8thupprim} is the following lemma.
\begin{lemma}
Let $0 \leq c \leq 1/100$, let $\chi$ be an even primitive Dirichlet character of modulus $q \in \mathbb{N}$, and let $t \in \mathbb{R}$. Then
\label{le:RamId}
\begin{equation}
\label{eq:RamId}
\begin{split}
L(\tfrac 12 + c+ it, \chi)^4 & = \sum_{n = 1}^{\infty} \frac{\tau_4(n) \chi(n) e^{-n / X}}{n^{1/2 + c+ it}} + F(\tfrac 12 + c + it) \sum_{n \leq X} \frac{\tau_4(n)\overline{\chi}(n)}{n^{1/2 - c - it}} \\ & - \frac{1}{2\pi i} \int_{(-3/4)} F(\tfrac 12 + c+ it + w) \Big ( \sum_{n > X} \frac{\tau_4(n) \overline{\chi}(n)}{n^{1/2-w - c  - it}} \Big ) \Gamma(w) X^w dw \\ 
& - \frac{1}{2\pi i}\int_{(1/4)} F(\tfrac 12 + c + it + w) \Big ( \sum_{n \leq X} \frac{\tau_4(n) \overline{\chi}(n)}{n^{1/2 - c -w- it}} \Big ) \Gamma(w) X^w dw \\
&- \mathbf{1}_{q = 1}  \operatorname*{Res}_{w=1/2 - c - it} \Gamma(w) X^{w}  \zeta^4\left(\frac 12 + c + it + w\right),
\end{split}
\end{equation}
where $\mathbf{1}_{q = 1}$ is $1$ if $q = 1$, and $0$ otherwise.
\end{lemma}
\begin{proof}
This follows as~\cite[Theorem 2]{Ramachandra}: Notice first that by Mellin inversion (since $\Gamma(s)$ is the Mellin transform of the function $e^{-x}$)
$$
\sum_{n = 1}^{\infty}  \frac{\tau_4(n)\chi(n) e^{-n / X}}{n^{1/2 + c+it}}  = \frac{1}{2\pi i} \int_{(2)} L(\tfrac 12 + c+ it + w, \chi)^4 \Gamma(w) X^w dw.
$$
We shift the integral to the line $\tRe w = - \tfrac 34$, collecting a residue from the pole at $w = 0$ and in case $q = 1$ also from a pole at $w = 1/2-c-it$. Applying also the functional equation to the integral on the line $\tRe w = - \tfrac 34$, we see that 
\[
\begin{split}
&\sum_{n = 1}^{\infty} \frac{\tau_4(n) \chi(n) e^{-n / X}}{n^{1/2 + c+it}}   = L(\tfrac 12 + c + it, \chi)^4 + \mathbf{1}_{q = 1} \operatorname*{Res}_{w=1/2 - c - it} \Gamma(w) X^{w}  \zeta^4\left(\frac 12 + c + it + w\right)\\
& + \frac{1}{2\pi i} \int_{(-3/4)} F(\tfrac 12 + c+ it + w, \chi) L(\tfrac 12 - c - it - w, \overline{\chi})^4 \Gamma(w) X^w dw.
\end{split}
\]
Writing on the last line
\[
L(\tfrac 12 - c - it - w, \overline{\chi})^4 = \sum_{n \leq X} \frac{\tau_4(n) \overline{\chi}(n)}{n^{1/2-c-it-w}} + \sum_{n > X} \frac{\tau_4(n) \overline{\chi}(n)}{n^{1/2-c-it-w}},
\]
the claim of the lemma reduces to the claim that
\[
\begin{split}
& \frac{1}{2\pi i} \int_{(-3/4)} F(\tfrac 12 + c+ it + w, \chi) \sum_{n \leq X} \frac{\tau_4(n) \overline{\chi}(n)}{n^{1/2-c-it-w}} \Gamma(w) X^w dw \\
&=  \frac{1}{2\pi i} \int_{(1/4)} F(\tfrac 12 + c+ it + w, \chi) \sum_{n \leq X} \frac{\tau_4(n) \overline{\chi}(n)}{n^{1/2-c-it-w}} \Gamma(w) X^w dw \\
& \qquad \qquad - F(\tfrac 12 + c+ it, \chi) \sum_{n \leq X} \frac{\tau_4(n)\overline{\chi}(n)}{n^{1/2 -c - it}}.
\end{split}
\]
But this is immediate from shifting the integration line, picking up a residue from a pole at $w=0$.
\end{proof}
Let us now return to~\eqref{eq:8thupprim} for even characters. Recall $Q \geq 3/2$ and $T \geq 3$. We take $X= (QT)^2$ and apply Lemma~\ref{le:RamId}, writing~\eqref{eq:RamId} as $L(\tfrac 12 +c+ it, \chi)^4 = (J_1 + J_2 - J_3 - J_4)(c, t, \chi)$ (the fifth term in~\eqref{eq:RamId} always vanishes as $q \geq 2$, so we do not need to include it). Then it suffices to show that
$$
\sum_{q \sim Q} \ \sumb_{\substack{\chi \pmod{q}}} \int_T^{2T} |J_k(c, t, \chi)|^2 dt \ll (QT)^2 (\log (QT))^{16}
$$
for $k = 1,2,3,4$. First, by the large sieve (Lemma~\ref{lem:largesieve}), we have
\[
\begin{split}
&\sum_{q \sim Q} \ \sumb_{\substack{\chi \pmod{q}}} \int_{T}^{2T} |J_1(c, t, \chi)|^2 dt \ll \sum_{n = 1}^\infty (Q^2 T + n) \frac{\tau_4(n)^2 e^{-2n/(Q^2T^2)}}{n^{1+2c}} \\
& \ll Q^2 T^2 \sum_{n \leq Q^2 T^2} \frac{\tau_4(n)^2}{n}  + \sum_{n > Q^2T^2} \tau_4(n)^2 e^{-2n/(Q^2T^2)} \ll Q^2T^2 (\log (QT))^{16}.
\end{split}
\]
Furthermore, noting that, by~\eqref{eq:Gammadiv}, $|F(\tfrac 12 + c + it)| \ll (q(1+|t|))^{-4c}$, the large sieve (Lemma~\ref{lem:largesieve}) implies that
\[
\begin{split}
\sum_{q \sim Q} \ \sumb_{\substack{\chi \pmod{q}}} \int_T^{2T} |J_2(c, t, \chi)|^2 dt 
 \ll (QT)^{2-8c} \sum_{n \leq Q^2 T^2} \frac{\tau_4(n)^2}{n^{1-2c}} \ll Q^2T^2 (\log (QT))^{16}.
\end{split}
\]

To deal with the remaining two cases we notice that, by~\eqref{eq:Gammadiv} and~\eqref{eq:Stirling}, we have, for $\tRe w \in \{-3/4, 1/4\}$, 
\[
|F(\tfrac 12 + it +c + w)| \ll q^{-4 \tRe w-4c} \cdot (1 + |t + \tIm w|))^{-4 \tRe w-4c} \quad \text{and} \quad |\Gamma(w)| \ll e^{-|\tIm w|}.
\]
Splitting into cases according to whether $|\tIm w| \leq \tfrac 12 |t|$ or not, we see that 
$$
|F(\tfrac 12 + c+ it + w) \Gamma(w) (QT)^{2w}| \ll (QT)^{-2 \tRe w-4c} e^{-\tfrac 12 |\tIm w|}.
$$
Hence by the Cauchy-Schwarz inequality, noting the rapid decay in $\tIm w$, and using the large sieve (\ref{lem:largesieve}) we have
\[
\begin{split}
&\sum_{q \sim Q} \ \sumb_{\substack{\chi \pmod{q}}} \int_T^{2T} |J_3(c, t, \chi)|^2 dt \\
&\ll (QT)^{3-8c} \sum_{q \sim Q} \ \sumb_{\substack{\chi \pmod{q}}} \int_T^{2T} \left(\int_{-\infty}^\infty \left|\sum_{n > Q^2T^2} \frac{\tau_4(n) \overline{\chi}(n)}{ n^{5/4-c-it-iu}}\right| e^{-u/2} du\right)^2 dt \\
&\ll (QT)^{3-8c} \max_u \sum_{q \sim Q} \ \sumb_{\substack{\chi \pmod{q}}} \int_T^{2T} \left|\sum_{n > Q^2T^2} \frac{\tau_4(n) \overline{\chi}(n)}{ n^{5/4-c-it-iu}}\right|^2 dt \\
& \ll (QT)^{3-8c} \sum_{n > Q^2 T^2} (Q^2 T + n) \frac{\tau_4(n)^2}{n^{5/2-2c}} \ll (QT)^2 (\log (QT))^{16}.
\end{split}
\]
Similarly
\[
\begin{split}
&\sum_{q \sim Q} \ \sumb_{\substack{\chi \pmod{q}}} \int_T^{2T} |J_4(c, t, \chi)|^2 dt \\ 
&\ll (QT)^{-1-8c} \max_u \sum_{q \sim Q} \ \sumb_{\substack{\chi \pmod{q}}} \int_T^{2T} \left|\sum_{n \leq Q^2T^2} \frac{\tau_4(n) \overline{\chi}(n) }{n^{1/4-c-it-iu}}\right|^2 dt \\
& \ll (QT)^{-1-8c} (Q^2 T + Q^2T^2)\sum_{n \leq Q^2 T^2} \frac{\tau_4(n)^2}{n^{1/2-2c}} \ll (QT)^2 (\log (QT))^{16}.
\end{split}
\]
\end{proof}

In case of the Riemann zeta function, a better upper bound is available, and it will be helpful in evaluating the main terms:
\begin{lemma}
\label{le:zeta8th}
Let $T \geq 3$. Then
\begin{equation}
\label{eq:zeta8th}
\int_{-T}^T |\zeta(1/2+c+it)|^8 dt \ll T^{3/2} (\log T)^{21/2}
\end{equation}
for any $c \geq 0$.
\end{lemma} 
\begin{proof}
Let us first consider the case $c = 0$. By the Cauchy-Schwarz inequality
\[
\int_{-T}^T |\zeta(1/2+it)|^8 dt \ll \left(\int_{-T}^T |\zeta(1/2+it)|^4 dt\right)^{1/2} \left(\int_{-T}^T |\zeta(1/2+it)|^{12} dt\right)^{1/2}
\]
Applying upper bounds for the fourth and twelfth power moments of the Riemann zeta function (see e.g.~\cite[formula (7.6.3)]{Ti} for the fourth moment and see~\cite{H-B12th} for the twelfth moment), this is 
\[
\ll \left(T (\log T)^4\right)^{1/2} \cdot \left(T^2 (\log T)^{17}\right)^{1/2},
\]
and the claim follows in case $c = 0$. In case $c \geq 1$ the left hand side of~\eqref{eq:zeta8th} is trivially bounded by $O(T)$. In the remaining case $c \in (0, 1)$ the claim follows from a convexity argument (see~\cite[Section 7.8]{Ti}).
\end{proof}

Let us here also record orthogonality relations for characters. There and later when $\pm$ appears only on one side of an equation, both options are summed.
\begin{lemma} \label{lem:orthogonal} If $m, n$ are integers with $(mn, q) = 1$ then
$$ \sumstar_{\chiq} \chi(m) \cb(n) =  \sum_{\substack{q = dr \\ r | (m - n)}} \mu(d) \phi(r),$$
and
$$ \sumb_{\chiq} \chi(m) \cb(n) = \frac{1}{2} \sum_{\substack{q = dr \\ r | (m \pm n)}} \mu(d) \phi(r).$$
\end{lemma}
\begin{proof} 
The first claim follows from the orthogonality of all characters and M\"obius inversion, while the second claim follows from the first by detecting even characters with $\frac{1+\chi(-1)}{2}$.
\end{proof}

\section{The approximate functional equation and truncation}
Writing
$$G(s, t) := \Gamma^4\left(\tfrac{s}{2} + \tfrac{it}{2}\right)\Gamma^4\left(\tfrac{s}{2}-\tfrac{it}{2}\right)$$
we would like to find an approximation to
\[
\int_{-\infty}^{\infty} |\Lambda(\tfrac{1}{2}+it, \chi)|^8 dt = \int_{-\infty}^{\infty} G\left(\tfrac 12, t\right) L^4 \left( \tfrac{1}{2} + it, \chi\right)L^4\left(\tfrac{1}{2} - it,  \cb \right) dt.
\]
Note that, for $ \tRe (s) > 1,$ we have
$$  L^4 \left( s + it, \chi\right)L^4\left(s - it,  \cb \right) = \sum_{m, n = 1}^\infty \frac{\tau_4 (m) \tau_4 (n)}{m^s n^s} \chi(m) \cb(n) \left( \frac{n}{m}\right)^{it}.$$   

For stating an approximate functional equation for the eighth moment, we need the weight function $V\colon \mathbb{R}_+^3 \to \mathbb{C}$ defined by
\begin{equation}
\label{eqnV}
 V(\xi, \eta; \mu) := \int_{-\infty}^{\infty} \left(\frac{\eta}{\xi}\right)^{it} W\left( \frac{\xi\eta \pi^4}{\mu^4}, t\right) \> dt
\end{equation}
with $W \colon \mathbb{R}_+ \times \mathbb{C} \to \mathbb{C}$ defined by
\begin{equation}
\label{eqnW}
W(x, t) := \frac{1}{2\pi i} \int_{(1)} G(1/2+s, t) x^{-s} \frac{ds}{s}.
\end{equation}
For the short proof of the following proposition, see~\cite[Lemma 2.1]{CL8th} (but note that the definition of $P(\chi, t)$ there is missing a factor $(n/m)^{it}$).
\begin{lemma} \label{prop:fnceqnLambda} 
Let $\chi$ be an even primitive character $\pmod{q}$. Then
\begin{equation}
\label{eq::fnceqnLambda}
\int_{-\infty}^{\infty} |\Lambda(1/2+iy, \chi)|^8 dy = 2 \sum_{m, n = 1}^{\infty} \frac{\tau_4(m) \tau_4(n)}{\sqrt{mn}} \chi(m) \overline{\chi}(n) V(m, n ;q).
\end{equation}
\end{lemma}

The integration in $y$ gives rise to $V(m, n; q)$ on the right hand side of~\eqref{eq::fnceqnLambda} which makes the essential support of $m$ and $n$ sums more restricted; from the following lemma, we see that the main contribution comes from when $m, n$ are both at most $q^{2 + \varepsilon}.$

\begin{lemma} \label{lem:weightWV} The weight function $W(x, t)$ defined in (\ref{eqnW}) is a smooth function of $x \in (0, \infty)$. 
Furthermore the function $V(\xi, \eta; \mu)$ defined in (\ref{eqnV}) satisfies, for any $\xi, \eta, \mu > 0$ and any non-negative integers $\nu_1, \nu_2, \nu_3$,
\begin{equation}
\label{eq:Vbound} 
\frac{d^{\nu_1} d^{\nu_2} d^{\nu_3}}{d\xi^{\nu_1} d\eta^{\nu_2} d\mu^{\nu_3}} V(\xi, \eta; \mu) \ll_{\nu_1, \nu_2, \nu_3} \exp \left(-\left(\frac{\max(\xi, \eta)^2}{\mu^4} \right) ^{1/4}\right) \cdot \frac{1}{\xi^{\nu_1} \eta^{\nu_2} \mu^{\nu_3}}.
\end{equation}
\end{lemma}

\begin{proof} The proof is essentially the same as \cite[Proof of Lemma 1]{CIS} but for completeness we prove~\eqref{eq:Vbound} here. Without loss of generality, we can assume that $\eta \geq \xi$. By definition, for any $c > 0$,
\[
 V(\xi, \eta; \mu) = \int_{-\infty}^\infty \left(\frac{\eta}{\xi}\right)^{it} \frac{1}{2\pi i} \int_{(c)} G(1/2+s, t) \left(\frac{\xi\eta \pi^4}{\mu^4}\right)^{-s} \frac{ds}{s} \> dt.
 \]
Substituting $it = z$, we obtain
\[
V(\xi, \eta; \mu) = - \frac{1}{2\pi} \int_{-i\infty}^{i\infty} \left(\frac{\eta}{\xi}\right)^{z} \int_{(c)} G(1/2+s, -iz) \left(\frac{\xi\eta \pi^4}{\mu^4}\right)^{-s} \frac{ds}{s} \> dz.
\]
We move the $z$-integration to the line $\tRe z = -c$. Writing $z = -c+it$ and $s = c+iu$, we obtain
\[
V(\xi, \eta; \mu) = \frac{1}{2\pi} \int_{-\infty}^{\infty} \left(\frac{\eta}{\xi}\right)^{-c+it} \int_{-\infty}^\infty G(1/2+c+iu, t - ic) \left(\frac{\xi\eta \pi^4}{\mu^4}\right)^{-c-iu} \frac{du}{c+iu} \> dt.
\]
Taking derivatives, we see that
\begin{align*}
\frac{d^{\nu_1} d^{\nu_2} d^{\nu_3}}{d\xi^{\nu_1} d\eta^{\nu_2} d\mu^{\nu_3}} V(\xi, \eta; \mu) &\ll_{\nu_1, \nu_2, \nu_3} \frac{1}{\xi^{\nu_1} \eta^{\nu_2} \mu^{\nu_3}} \int_{-\infty}^{\infty} \int_{-\infty}^\infty |G(1/2+c+iu, t - ic)| \left(\pi \frac{\eta^{1/2}}{\mu}\right)^{-4c} \\
& \qquad \cdot |t+u|^{\nu_1} |2c+i(u-t)|^{\nu_2} |c+iu|^{\nu_3} \frac{du}{|c|+|u|} \> dt.
\end{align*}
When $\eta^{1/2}/\mu \leq 10$ we take $c = 2$ and the result follows immediately. Otherwise we choose $c = \eta^{1/2}/\mu > 10$. Then, by Stirling's formula~\eqref{eq:Stirling}, 
\begin{equation}
\label{eq:Gpicbound}
\begin{split}
&|G(1/2+c+iu, t + ic)| \left(\pi \frac{\eta^{1/2}}{\mu}\right)^{-4c} = \left|\Gamma\left(\frac{1}{4} +i \frac{u+t}{2}\right)^4 \Gamma\left(\frac{1}{4} + c +i \frac{u-t}{2}\right)^4\right| (\pi c)^{-4c} \\
& \ll \exp(-|u+t|)\frac{\left|\frac{1}{4} + c + i \frac{u-t}{2}\right|^{4c}}{(e \pi c)^{4c}} \exp\left(-2 |u-t| \arctan\left(\frac{|u-t|}{2c+1/2}\right)\right) 
\end{split}
\end{equation}
If $|u-t| \leq 2c$, then~\eqref{eq:Gpicbound} is
\[
\ll \exp(-|u+t|)\frac{(2c)^{4c}}{(e\pi c)^{4c}} \ll \exp(-|u+t| - 4c) \ll \exp(-|u+t| - |u-t| - 2c).
\]
On the other hand if $|u-t| > 2c$, then~\eqref{eq:Gpicbound} is
\[
\begin{split}
&\ll \exp(-|u+t|)\frac{\left|u-t\right|^{4c}}{(e\pi c)^{4c}} \exp\left(-2 |u-t| \cdot \frac{3}{4}\right) \\
&\ll \exp(-|u+t|-|u-t|) \exp\left(-c \left(4\log (e\pi) + \frac{1}{2} \cdot \frac{|u-t|}{c} - 4 \log\frac{|u-t|}{c}\right)\right) \\
&\ll \exp(-|u+t|-|u-t| -2c). 
\end{split}
\]
Hence in any case
\[
\frac{d^{\nu_1} d^{\nu_2} d^{\nu_3}}{d\xi^{\nu_1} d\eta^{\nu_2} d\mu^{\nu_3}} V(\xi, \eta; \mu) \ll_{\nu_1, \nu_2, \nu_3} \frac{e^{-c}}{\xi^{\nu_1} \eta^{\nu_2} \mu^{\nu_3}} \int_{-\infty}^{\infty} \int_{-\infty}^\infty \exp(-\tfrac{1}{2}(|u+t|+|u-t|)) du \> dt
\]
and the claim follows.
\end{proof}

Let us make some initial reductions to the left hand side of our claim~\eqref{eq:ThmClaim}. First, by Lemma \ref{prop:fnceqnLambda}
$$\sum_{q } \Psi\left(\frac{q}{Q}\right) \sumb_{\chiq} \int_{-\infty}^{\infty} \left| \Lambda\left( \tfrac{1}{2} + iy, \chi \right)\right|^8 \> dy = 2\Delta(\Psi, Q),$$
where
$$ \Delta(\Psi, Q) = \sum_{q} \Psi \left(\frac{q}{Q}\right) \sumb_{\chiq}  \sum_{m, n = 1}^{\infty} \frac{\tau_4(m) \tau_4(n)}{\sqrt{mn}} \chi(m) \overline{\chi}(n) V(m, n ;q).$$

The next step is to slightly truncate the sums over $m$ and $n$ in $\Delta(\Psi, Q)$. This truncation will allow us to apply the complementary divisor trick to reduce the conductor in Section \ref{secG}. For the truncation, let $\varepsilon_0 > 0$ and define
\begin{equation} \label{def:Q0}
Q_0 := \exp((\log Q)^{\varepsilon_0})
\end{equation}
and
 $$\dt (\Psi, Q) = \sum_{q} \sumb_{\chiq} \Psi \left(\frac{q}{Q}\right) \sum_{m, n = 1}^{\infty} \frac{\tau_4(m) \tau_4(n)}{\sqrt{mn}} \chi(m) \overline{\chi}(n) V\left(m, n ;\frac{q}{\lon}\right).$$ 
Note that we expect (and will later show) that $\Delta (\Psi, Q) \asymp Q^2(\log Q)^{16}.$ In this section we recall from~\cite{CL8th} that $\dt (\Psi, Q)$ is a sufficiently close approximation to $\Delta (\Psi, Q)$. 
This type of procedure has previously appeared in the contexts of moments also in other situations (see for instance the works of Soundararajan \cite{SoundDiriMoment} and Soundararajan and Young \cite{SY}). 

\begin{prop} \label{prop:truncation} With the above notation 
$$ {\Delta}(\Psi, Q) - \dt\left(\Psi, Q \right) \ll_{\alpha} Q^2 (\log Q)^{15+\varepsilon_0}.$$
\end{prop}

\begin{proof} This follows as~\cite[Proof of Proposition 3.1]{CL8th}. The key ingredient in the proof is the large sieve. In~\cite{CL8th} one has $(\log Q)^\alpha$ in place of $Q_0$ but the same proof works by replacing every occurrence of $(\log Q)^\alpha$ by $Q_0$ --- our choice of $Q_0$ is sufficiently small since in the critical case of~\cite[Proof of Proposition 3.1]{CL8th} where $Q^2/Q_0^2 \leq M, N \leq Q^2$ the parameters $M$ and $N$ are each summed over $\ll \log Q_0 \ll (\log Q)^{\varepsilon_0}$ dyadic intervals. 
\end{proof}

\section{Splitting $\dt\left(\Psi, Q \right)$}
Now, by Proposition~\ref{prop:truncation} and orthogonality of characters (see Lemma \ref{lem:orthogonal}), it is sufficient to consider
$$\dt\left(\Psi, Q\right) = \h \sum_{m, n=1}^\infty \frac{\tau_4(m) \tau_4(n)}{\sqrt{mn}}\sum_{\substack{d, r\\(dr, mn) = 1 \\ r|m\pm n}} \phi(r) \mu(d) \Psi\bfrac{dr}{Q} V\left(m, n, \frac{dr}{\lon}\right).$$
Let $D = (\log Q)^{\Delta_0}$ for some $\Delta_0$ to be determined later and split 
\[
\dt\left(\Psi, Q\right) = \D  + \Smo + \mathcal{O}(\Psi, Q),
\]
where the diagonal term $\D$ consists of the terms with $m=n$, the term $\Smo$ consists of the remaining terms with $d>D$ and $\mathcal{O}(\Psi, Q)$ consists of the remaining terms with $d \leq D$. 

The first two terms $\D$ and $\Smo$ were handled in~\cite[Propositions 4.1 and 5.1]{CL8th} that gave the following asymptotic formulas.
\begin{lemma}
Let $\varepsilon > 0$ and let $\widetilde{\Psi}$ be the Mellin transform of $\Psi,$ which is defined by \begin{equation} \label{eqn:MellinPsi}
 \widetilde{\Psi} (s) = \int_{0}^\infty \Psi(u) u^s \frac{du}{u}.
\end{equation}
Then 
\[
\begin{split}
\D &= 2^{16}Q^2 \frac{(\log Q )^{16}}{16!} \widetilde{\Psi}(2) \frac{\mathcal{A}}{2} \prod_p\left( 1 - \frac{1}{p}\right)\left(1 + \frac{1}{\mathcal{B}_p} \left( \frac{1}{p} - \frac{1}{p^2} - \frac{1}{p^3}\right)\right) \int_{-\infty}^{\infty} G(1/2, t)dt \\
&\qquad \qquad + O_{\varepsilon}(Q^2(\log Q)^{15+\varepsilon}),
\end{split}
\]
where 
\[
\mathcal{B}_p = \sum_{r = 0}^{\infty} \frac{\tau_4^2(p^r)}{p^{r}} \quad \text{and} \quad \mathcal{A} = \prod_p \mathcal{B}_p \left(1 - \frac{1}{p} \right)^{16}.
\]
\end{lemma}

\begin{lemma} \label{lem:sumSm1} Let $\varepsilon > 0$. There exists an absolute constant $C$ such that
$$ \Smo = \mathcal {MS}_1(\Psi, Q) + O_{\varepsilon}\left(\frac{Q^2 (\log Q)^{C}}{D^{1 - \varepsilon}}\right),$$
where 
\[
 \mathcal {MS}_1(\Psi, Q) := - \sum_{\substack{m, n = 1  \\ m \neq n}} \frac{\tau_4 (m) \tau_4 (n)}{\sqrt{mn}} \sum_{(q, mn) = 1} \Psi \left( \frac{q}{Q}\right)\left( \sum_{\substack{dr = q \\ d \leq D}} \mu(d) \right)  V\left(m, n ; \frac{q}{\lon} \right).
\]
\end{lemma}

In~\cite{CL8th} the remaining term $\mathcal{O}(\Psi, Q)$ was treated by the complementary divisor trick. In doing this, the first step is to replace the arithmetic factor $\phi(r)$ by a smooth function which can be done by writing $\phi(r) = \sum_{a l = r} \mu(a) l$ so that
\[
\mathcal{O}(\Psi, Q) = \h \sum_{\substack{m, n=1 \\ m \neq n}}^\infty \frac{\tau_4(m) \tau_4(n)}{\sqrt{mn}}\sum_{\substack{d \leq D, a, l\\(dal, mn) = 1 \\ al | m\pm n}} \mu(a) \mu(d) l \Psi\bfrac{adl}{Q} V\left(m, n, \frac{adl}{\lon}\right).
\]
It turns out that we can only use the complementary divisor trick when $a$ is not too large (this is ultimately due to Proposition~\ref{prop:8momentIntandSumoverq} involving a $q$-average, as explained in Remark~\ref{rem:aleqA} below). Accordingly we write
\[
\begin{split}
\mathcal{O}(\Psi, Q) &= \Smtwo + \mathcal{G}(\Psi, Q),
\end{split}
\] 
where $\Smtwo$ consists of the terms with $a > A := \exp((\log Q)^{\varepsilon_0/2})$ 
and $\G$ consists of the terms with $a \leq A$.

To handle $\Smtwo$, we prove the following lemma. 
\begin{lemma} \label{lem:sumSm2} Let $\varepsilon > 0$. One has
$$ \Smtwo = \mathcal {MS}_2(\Psi, Q) + O_{\varepsilon}\left(\frac{Q^2 \exp(2(\log Q)^{\varepsilon/3})}{A^{\varepsilon/3}}\right),$$
where 
\[
\mathcal {MS}_2(\Psi, Q) := \sum_{\substack{m, n = 1  \\ m \neq n}} \frac{\tau_4 (m) \tau_4 (n)}{\sqrt{mn}} \sum_{(q, mn) = 1} \Psi \left( \frac{q}{Q}\right)\left( \sum_{\substack{adl = q \\ a > A, d \leq D}} \frac{\mu(a) \mu(d) l}{\phi(al)} \right)  V\left(m, n ; \frac{q}{\lon} \right).
\]
\end{lemma}

\begin{proof}
The proof mostly follows the proof of~\cite[Proposition 5.1]{CL8th} (i.e. proof of Lemma~\ref{lem:sumSm1}).
Recall that
$$\Smtwo = \h \sum_{\substack{m, n=1 \\ m \neq n}}^\infty \frac{\tau_4(m) \tau_4(n)}{\sqrt{mn}}\sum_{\substack{a > A, d \leq D, l\\(dal, mn) = 1 \\ al | m\pm n}} \mu(a) \mu(d) l \Psi\bfrac{adl}{Q} V\left(m, n, \frac{adl}{\lon}\right).
$$
We express the condition $al | m \pm n$ using the even characters $\chi \pmod{al}.$ Hence
\[ 
\Smtwo = \sum_{\substack{a > A \\ d \leq D, l}} \frac{\mu(a) \mu(d) l}{\phi(al)} \Psi\bfrac{adl}{Q} \sum_{\substack{\chi \pmod{al} \\ \chi(-1) = 1}}  \sum_{\substack{m, n=1 \\ m \neq n \\(mn, d) = 1 }}^\infty \frac{\chi(m)\overline{\chi}(n)\tau_4(m) \tau_4(n)}{\sqrt{mn}}\V\left(m, n, \frac{adl}{\lon}\right).
\]
The principal character $\chi = \chi_0$ gives the claimed main term, so we can concentrate on the contribution of the non-principal characters. 

Reintroducing the terms with $m=n$ leads to (here our treatment is simpler than that in~\cite[Proof of Proposition 5.1]{CL8th}) an error at most
\[
\sum_{n=1}^\infty \frac{\tau_4^2(n)}{n}\sum_{\substack{a > A \\ d \leq D, l \\ (dal, n) = 1}} l \Psi\bfrac{adl}{Q} \left| \V\left(n, n, \frac{adl}{\lon}\right)\right|.
\]
Recalling the support of $\Psi$ and Lemma~\ref{lem:weightWV} this is
\[
\sum_{n \leq Q^3} \frac{\tau_4^2(n)}{n} \sum_{A < a \leq 2Q} \sum_{d \leq D} \sum_{l \leq \frac{2Q}{ad}} l + O\left(\frac{1}{Q}\right) \ll \frac{Q^2 (\log Q)^{16}}{A}.
\]
Hence it suffices to bound the sum
\begin{align*}
\mathcal{S}_2'(\Psi, Q) &= \sum_{\substack{a > A \\ d \leq D, l}} \frac{\mu(a) \mu(d) l}{\phi(al)} \Psi\bfrac{adl}{Q} \sum_{\substack{\chi \pmod{al} \\ \chi(-1) = 1 \\ \chi \neq \chi_0}}  \sum_{\substack{m, n=1 \\(mn, d) = 1 }}^\infty \frac{\chi(m)\overline{\chi}(n)\tau_4(m) \tau_4(n)}{\sqrt{mn}}\V\left(m, n, \frac{adl}{\lon}\right)
\end{align*}
Arguing similarly to the proof of Proposition 5.1 in \cite{CL8th} but with $al$ in place of $r$, with not deducing to primitive characters, and with an additional factor $\mu(a)l/\phi(al) \ll \log Q/\phi(a)$, we get that
\begin{equation}
\label{eq:S2'bound}
\begin{split}
\mathcal{S}_2'(\Psi, Q) &\ll_{\varepsilon} (\log Q) \sum_{d \leq D} d^{\varepsilon} \int_{-\infty}^{\infty} \int_{-\infty}^{\infty} \exp(-|t_1| - |t_2|)  \sum_{a > A} \frac{1}{\phi(a)} \sum_{\substack{l \leq \frac{2Q}{ad} }}  1 \\
& \cdot \sum_{\substack{\chi \ \textrm{mod} \ al \\ \chi(-1) = 1 \\ \chi \neq \chi_0}} \left\{ \left| L\left(\tfrac{1}{2}+ \tfrac{1}{\log Q} + it_1,  \chi\right)\right|^8 + \left|L\left(\tfrac{1}{2}+ \tfrac{1}{\log Q} + it_2,  \overline{\chi}\right)\right|^8 \right\} \> dt_1 \> dt_2.
\end{split}
\end{equation}
Now, writing $al = r$,
\[
\sum_{a > A} \frac{1}{\phi(a)} \sum_{\substack{l \leq 2Q/(ad) }} 1 \leq \sum_{r \leq 2Q/d}  \sum_{\substack{a \mid r \\ a > A}} \frac{1}{\phi(a)}.
\]
Here
\begin{equation}
\label{eq:aphibound}
\begin{split}
\sum_{\substack{a \mid r \\ a > A}} \frac{1}{\phi(a)} &\ll \frac{\log Q}{A^{\varepsilon/3}} \sum_{a \mid r} \frac{1}{a^{1-\varepsilon/3}} \ll \frac{\log Q}{A^{\varepsilon/3}} \prod_{p \mid r} \left(1+\frac{1}{p^{1-\varepsilon/3}}\right) \\ &\ll \frac{\log Q}{A^{\varepsilon/3}} \prod_{p \leq \log Q} \left(1+\frac{1}{p^{1-\varepsilon/3}}\right) \ll \frac{\log Q}{A^{\varepsilon/3}} \exp((\log Q)^{\varepsilon/3}),
\end{split}
\end{equation}
Using this and Proposition~\ref{prop:8momentIntandSumoverq} in~\eqref{eq:S2'bound}, we get 
\begin{align*}
\mathcal{S}_2'(\Psi, Q) &\ll_{\varepsilon} \frac{\exp(\frac 32(\log Q)^{\varepsilon/3})}{A^{\varepsilon/3}} \sum_{d \leq D} d^\varepsilon \max_{R \leq \frac{2Q}{d}} \int_{-\infty}^{\infty} \exp(-|t_1|) \\
& \hskip 1.5in \cdot  \sum_{r \sim R}  \sum_{\substack{\chi \textrm{mod} \ r \\ \chi(-1) = 1}} \left| L\left(\tfrac{1}{2}+ \tfrac{1}{\log Q} + it_1,  \chi\right)\right|^8 \> dt_1 \\
&\ll_{\varepsilon} \frac{\exp(\frac 32(\log Q)^{\varepsilon/3})}{A^{\varepsilon/3}} \sum_{d \leq D} d^\varepsilon \max_{R \leq \frac{2Q}{d}} R^2 (\log Q)^{16} \ll \frac{Q^2\exp(2(\log Q)^{\varepsilon/3})}{A^{\varepsilon/3}} , 
\end{align*}
as claimed. 
\end{proof}

\begin{rem}
The bound~\eqref{eq:aphibound} is the reason we need to make a rather large choice like $A = \exp((\log Q)^{\varepsilon_0/2})$ which in turn in the treatment of $\mathcal{G}(\Psi, Q)$ will force a choice like $Q_0 = \exp((\log Q)^{\varepsilon_0})$. 
\end{rem}
Combining the previous two lemmas we obtain
\begin{prop} \label{prop:sumS} Let $\varepsilon > 0$. There exists an absolute constant $C$ such that
$$ \Smo + \Smtwo = \mathcal {MS}(\Psi, Q) +  O_{\varepsilon}\left(\frac{Q^2 (\log Q)^{C}}{D^{1 - \varepsilon}} + \frac{Q^2 \exp(2(\log Q)^{\varepsilon/3})}{A^{\varepsilon/3}}\right),$$ 
where 
\begin{equation}
\label{eq:MSdef}
\mathcal{MS}(\Psi, Q) := - \sum_{\substack{m, n = 1  \\ m \neq n}} \frac{\tau_4 (m) \tau_4 (n)}{\sqrt{mn}} \sum_{(q, mn) = 1} \Psi \left( \frac{q}{Q}\right)\left( \sum_{\substack{adl = q \\ a \leq A, d \leq D}} \frac{\mu(a) \mu(d) l}{\phi(al)} \right)  V\left(m, n ; \frac{q}{\lon} \right).
\end{equation}
\end{prop}
\begin{proof}
Writing the sum over $d$ in $\mathcal{MS}_1(\Psi, Q)$ as 
\[
\sum_{\substack{dr = q \\ d \leq D}} \mu(d) = \sum_{\substack{dr = q \\ d \leq D}} \mu(d) \sum_{r = a l} \frac{\mu(a) l}{\phi(r)} = \sum_{\substack{adl = q \\ d \leq D}} \frac{\mu(a) \mu(d) l}{\phi(al)}
\]
we see that
\[
\mathcal{MS}_1(\Psi, Q) + \mathcal{MS}_2(\Psi, Q) = \mathcal{MS}(\Psi, Q),
\]
so the claim follows from Lemmas~\ref{lem:sumSm1} and~\ref{lem:sumSm2}.
\end{proof}
Now we have asymptotic formulas for $\D$ and for $\Smo + \Smtwo$. In the following section we will turn our attention to $\G$.
\section{Treatment of $\G$}\label{secG}
\subsection{The complementary divisor}
Recall that
\begin{equation}
\label{eq:Gdefrep}
\G = \h \sum_{\substack{m, n=1 \\ m \neq n}}^\infty \frac{\tau_4(m) \tau_4(n)}{\sqrt{mn}}\sum_{\substack{a \leq A, d \leq D, l\\(dal, mn) = 1 \\ al | m\pm n}} \mu(a) \mu(d) l \Psi\bfrac{adl}{Q} V\left(m, n, \frac{adl}{\lon}\right).
\end{equation}
We write $g=(m, n)$ and $m=gM$, $n=gN$, so that $(M, N) = 1$. Note that the condition $(al, mn) = 1$ can be replaced by $(al, g) = 1$ since $al \mid m \pm n$. Necessarily $al \mid M \pm N$ and we write $|M\pm N| = alh$.  We want to replace the condition modulo $al$ with a condition modulo $ah$, which will be small when $l$ is large. To do so, we express the condition $(l, g) = 1$ by $\sum_{b|(l, g)} \mu(b)$.  Writing $l = bk$, the inner sum in~\eqref{eq:Gdefrep} becomes
$$\sum_{\substack{d\leq D\\(d, gMN) = 1}} \mu(d) \sum_{\substack{a \leq A \\(a, g) = 1}} \mu(a) \sum_{b|g}\mu(b) \sum_{\substack{k\geq 1\\|M\pm N| = abkh}}bk \Psi\bfrac{dabk}{Q} \V\left(gM, gN, \frac{dabk}{\lon }\right).
$$
Substituting $k=\frac{|M \pm N|}{abh}$ and rearranging, this equals
\begin{align} \label{eqn:stepbeforeW}
Q \sum_{\substack{d\leq D \\ (d, gMN) = 1}}  \sum_{\substack{a \leq A \\ (a, g)=1}} \sum_{b|g} \sum_{\substack{h > 0 \\ M \equiv \mp N \ ({\rm mod} \ abh )}} & \frac{\mu(a) \mu(b) \mu(d)}{ad} \\
& \cdot \left(\frac{d|M\pm N|}{Qh} \right)\Psi \bfrac{d|M \pm N|}{Qh} \V\left(gM, gN; \frac{d|M \pm N|}{h \lon }\right). \nonumber
\end{align}
For non-negative real numbers $u, x, y$ and for each choice of sign, we define
\begin{equation} \label{def:mathcalWpm}\W^\pm(x, y; u) = u|x \pm y|\Psi(u|x\pm y|) \V(x, y; u|x\pm y|).
\end{equation}
and
\[
\W(x, y; u) = \W^+(x, y; u) + \W^-(x, y; u)
\]
It is immediate from the definition of $\V(m, n; \mu)$ that, for any $c > 0$,
\begin{equation} \label{eqn:Vc}
\V(cm, cn; \sqrt c \mu) 
= \int_{-\infty}^\infty \bfrac{n}{m}^{it} W\left( \frac{mn\pi^4}{\mu^4}, t\right) dt  = \V (m, n; \mu).
\end{equation}
Thus 
\begin{align*}
\G &= \frac{Q}{2} \sum_{\substack{m, n=1 \\ m \neq n}}^\infty \frac{\tau_4(m) \tau_4(n)}{\sqrt{mn}} \sum_{\substack{d\leq D \\ (d, gMN) = 1}}  \sum_{\substack{a \leq A \\ (a, g)=1}} \sum_{b|g} \sum_{\substack{h > 0 \\ M \equiv \mp N \ ({\rm mod} \ abh )}}  \frac{\mu(a) \mu(b) \mu(d)}{ad} \\
& \qquad \qquad \cdot \W^\pm \left(\frac{gM \lotwo}{Q^2}, \frac{gN \lotwo}{Q^2} ; \frac{Qd}{gh \lotwo}\right). 
\end{align*}
Note that since $(M, N) = 1$, necessarily $(MN, abh) = 1.$ 

We express the condition $M \equiv \mp N (\ {\rm mod} \ abh)$ using characters $\chi$ (mod $abh$).  We then separate the principal character contribution, which is the main term. Specifically, we write
$$ \mathcal{G}(\Psi, Q) = \mathcal{MG}(\Psi, Q) + \mathcal{EG}(\Psi, Q),$$
where 
\begin{align} \label{eqn:maintermG}
\mathcal{MG}(\Psi, Q) := \frac{Q}{2} \sum_{\substack{m, n = 1\\ m \neq n}}^{\infty} & \frac{\tau_4(m) \tau_4(n)}{\sqrt{mn}}\sum_{\substack{d \leq D \\ (d, gMN) = 1}}  \sum_{\substack{a \leq A \\ (a, gMN)=1}} \sum_{\substack{b|g \\ (b, MN) = 1}} \sum_{\substack{h > 0 \\ (h, MN) = 1}}  \frac{\mu(a) \mu(b) \mu(d)}{ad \phi(abh)}\\
&\cdot  \W^\pm \left(\frac{gM \lotwo}{Q^2}, \frac{gN \lotwo}{Q^2} ; \frac{dQ}{gh \lotwo}\right), \nonumber
\end{align}
and
\begin{equation}
\label{eqn:errortermG}
\begin{split}
\mathcal{EG}(\Psi, Q) :=  \frac{Q}{2} \sum_{\substack{m, n = 1\\ m \neq n}}^{\infty} & \frac{\tau_4(m) \tau_4(n)}{\sqrt{mn}}\sum_{\substack{d \leq D \\ (d, gMN) = 1}}  \sum_{\substack{a \leq A \\ (a, gMN)=1}} \sum_{\substack{b|g \\ (b, MN) = 1}} \sum_{\substack{h > 0 \\ (h, MN) = 1}}  \frac{\mu(a) \mu(b) \mu(d) }{ad \phi(abh)} \\
&\cdot \sum_{\substack{\chi \ ({\rm mod} \ abh) \\ \chi \neq \chi_0}} \chi(M) \cb(\mp N) \W^\pm \left(\frac{gM \lotwo}{Q^2}, \frac{gN \lotwo}{Q^2} ; \frac{dQ}{gh \lotwo}\right).
\end{split}
\end{equation}
\subsection{Mellin transforms of $\W^\pm$} \label{sec:Mellin} 
To evaluate $\mathcal{MG}(\Psi, Q)$ and $\mathcal{EG}(\Psi, Q)$, we will write $\W^\pm(x, y; u)$ in terms of its Mellin transforms. We shall consider three different types of Mellin transforms. They come from taking Mellin transforms in the variable $u$ when we need to sum over the modulus $h$, the variables $x$ and $y$ when we need to sum over $M$ and $N$, and in all three variables when we need to sum over $M$, $N$, and $h$.  (In the description above, we have neglected to mention the conceptually less important sums over $d$, $a$, $b$ and $g$.)  

We collect the properties of the various Mellin transforms in the following three lemmas. The first lemma is from~\cite[Section 6.2]{CL8th} and the proof is the same as in \cite[Section 7]{CIS}, but using the bounds of Lemma \ref{lem:weightWV} in place of~\cite[Lemma 1]{CIS}. 
\begin{lemma} \label{lem:MellinU}
Given positive real numbers $x$ and $y$, define
$$ \Wt^\pm_1(x, y; z) = \int_0^\infty \W^\pm (x, y ; u) u^{z} \frac{du}{u}.$$
Then the functions $\Wt^\pm_1(x, y; z)$
are analytic for all $z \in \mathbb C$. For any $c \in \mathbb{R}$, we have the Mellin inversion formula
\[
 \W^\pm(x, y; u) = \frac{1}{2\pi i} \int_{(c)} \Wt^\pm_1(x, y; z) u^{-z} \> dz.
\]
For any non-negative integer $\nu$, any real numbers $x, y > 0$ and any $z \in \mathbb{C}$ one has
$$ |\Wt^\pm_1(x, y;z)| \ll_\nu |x \pm y|^{-\tRe z} \prod_{j = 1}^\nu |z + j|^{-1} \exp \left(-c_1 \max (x, y)^{1/4}\right)$$
for some absolute constant $c_1.$
\end{lemma}

The next lemma is similar to Lemma 6.3 in \cite{CL8th}, but we include a bound for the error resulting from truncating the integrals over $s_1$ and $s_2$. In both \cite{CL8th} and \cite{CIS}, the truncation was not explained. We include it here for completeness. 

\begin{lemma} \label{lem:MellinXYU} 

Define
$$ \Wt^\pm_3 (s_1, s_2; z) = \int_0^\infty \int_0^\infty \int_0^\infty \W^\pm(x, y ; u) u^z x^{s_1}y^{s_2} \frac{du}{u} \frac{dx}{x} \frac{dy}{y},$$
and 
$$ \Wt_3 (s_1, s_2; z) = \Wt^+_3 (s_1, s_2; z) + \Wt^-_3 (s_1, s_2; z).$$

Let $\omega = \frac{s_1 + s_2 - z}{2}$ and $\xi = \frac{s_1 - s_2 + z}{2}.$ For $\tRe(s_1), \tRe(s_2) > 0,$ and $|\tRe(s_1 -s_2)| < \tRe(z) < 1$ we have
\begin{equation} \label{eqn:mellinthree}
\Wt_3(s_1, s_2; z) = \frac{\psit(1 + 4\omega + z)}{2\omega \pi^{4\omega}} \int_{-\infty}^{\infty} \Hc (\xi - it, z) G\left(\h + \omega, t \right) \> dt, 
\end{equation}
where $\psit$ is defined in (\ref{eqn:MellinPsi}), and 
\begin{align}\label{eqn:Huv}
 \Hc (u, v) 
 &= \frac{\Gamma(u)\Gamma(v-u)}{\Gamma(v)}+ \frac{\Gamma(u) \Gamma(1-v)}{\Gamma(1+u-v)} + \frac{\Gamma(v-u) \Gamma(1-v)}{\Gamma(1-u)}  \\
&=\pi^{1/2} \frac{\Gamma\left(\tfrac{u}{2} \right)\Gamma\left(\tfrac{1-v}{2} \right)\Gamma\left(\tfrac{v-u}{2} \right)}{\Gamma\left(\tfrac{1-u}{2} \right)\Gamma\left(\tfrac{v}{2} \right)\Gamma\left(\tfrac{1-v + u}{2} \right)} \notag.
\end{align}

Let $x \neq y$ and $T \ge Q^{\varepsilon}$. For any $c_1, c_2 > 0$ with $ |c_1 - c_2| < c < 1$, one has the truncated Mellin inversion formulas
\begin{align}
\label{eqn:truncate}
\W(x, y ; u) &=\frac{1}{(2\pi i)^3} \int_{(c)} \int_{c_1 - iT}^{c_1 + iT}\int_{c_2 - iT}^{c_2 + iT} \Wt_3 (s_1, s_2 ; z) u^{-z} x^{-s_1} y^{-s_2} \>d s_2 \> d s_1 \> dz \\
\nonumber
&\qquad \qquad \qquad \qquad +  O\left(\frac{ u^{-c} x^{-c_1} y^{-c_2}}{ T^{1 - c}\left|\log\left( \frac xy\right) \right| }  \right).
\end{align}Moreover, let $\Wt_1(x, y ; z) = \Wt_1^+(x, y ; z) + \Wt_1^-(x, y ; z)$. Then for $\tRe z = c$,
\begin{align}
\label{eqn:truncate2}
\Wt_1(x, y ; z) &=\frac{1}{(2\pi i)^2} \int_{c_1 - iT}^{c_1 + iT}\int_{c_2 - iT}^{c_2 + iT} \Wt_3 (s_1, s_2 ; z) x^{-s_1} y^{-s_2} \>d s_2 \> d s_1 \> \\
\nonumber
&\qquad \qquad \qquad \qquad +  O\left(\frac{  x^{-c_1} y^{-c_2}}{ T^{1 - c}\left|\log\left( \frac xy\right) \right| }  \right).
\end{align}
Finally, the Mellin transform $\Wt_3(s_1, s_2; z)$ satisfies the bound
\begin{equation} \label{eqn:boundWt3}
|\Wt_3(s_1, s_2;z)| \ll (1 + |z|)^{-A} (1 + |\omega|)^{-A} (1 + |\xi|)^{\tRe(z) - 1}.
\end{equation}

\end{lemma}

\begin{proof} All claims but~\eqref{eqn:truncate} and \eqref{eqn:truncate2} are from \cite[Lemma 6]{CIS}.  By symmetry we can assume that $x < y$.  We note that by (50) - (52) in \cite{CIS} we have that for $0< b_3 < \tRe(z) < 1$, 
\begin{equation} \label{eqn:betaintegral}
|x + y|^{-z} + |x - y|^{-z} =  \frac{y^{-z}}{2\pi i} \int_{(b_3)} \mathcal H(w, z) \left( \frac{x}{y} \right)^{-w}  dw.
\end{equation}
\\
By \eqref{eqn:mellinthree}, Mellin inversion, and a change of variables,
\begin{align}
\label{eq:Wint}
 \W(x, y; u) 
&= \frac{1}{(2\pi i)^3} \int_{(b_1)} \int_{(c)} \int_{-\infty}^{\infty} u^{-z} \widetilde{\Psi}(1 + 4s + z) \int_{(b_3)} \mathcal H(w, z) \left( \frac{x}{y} \right)^{-w}  dw \\
\nonumber
 &\qquad \qquad \qquad \cdot  G\left(\frac 12 + s, t \right) \frac{y^{-s - z + it} x^{-s - it}}{\pi^{4s}}  \> dt \> dz \> \frac{ds}{s},
\end{align}and similarly

\begin{align}
\label{eq:Wint2}
 \Wt_1(x, y; z) 
&= \frac{1}{(2\pi i)^2} \int_{(b_1)} \int_{-\infty}^{\infty}  \widetilde{\Psi}(1 + 4s + z) \int_{(b_3)} \mathcal H(w, z) \left( \frac{x}{y} \right)^{-w}  dw \\
\nonumber
 &\qquad \qquad \qquad \cdot  G\left(\frac 12 + s, t \right) \frac{y^{-s - z + it} x^{-s - it}}{\pi^{4s}}  \> dt \>  \frac{ds}{s}.
\end{align}

We now give a proof of  the truncated Mellin inversion formula \eqref{eqn:truncate}, the proof of \eqref{eqn:truncate2} following by the same lines. 

Note that the integrals over $t$, $z$ and $s$ decay rapidly along vertical lines due to the factors $\widetilde{\Psi}$ and $G\left( \frac 12 + s, t\right)$. The integrals over $t, z$ and $s$ can thus be truncated to short integrals of length $T^{1/2} \geq Q^{\varepsilon/2}$, up to an error of 
\[
\frac{u^{-c} x^{-b_3-b_1} y^{-b_1-c+b_3}}{T^{100}}.
\] 

Now, consider $z$ fixed with $|\tIm z| \leq Q^{\varepsilon/2}$ and $x < y$. 
By Stirling's formula for Gamma functions in $\mathcal H(w,z)$, for $\tRe z = c$, we have that 
\begin{equation}\label{eqn:horizonalbdd}
\left| \left( \int_{b_3 + iT}^{-\infty + iT}+ \int_{-\infty + iT}^{b_3 + iT} \right) \mathcal H(w, z) \left( \frac xy\right)^{-w} \> dw\right| \ll \frac{1}{T^{1 - c}} \int_{-\infty}^{b_3} \left( \frac xy\right)^{-\sigma} \> d\sigma \ll \frac{x^{-b_3} y^{b_3}}{T^{1 - c} \left| \log \left( \frac xy \right)\right|}.
\end{equation}
In the above, whenever $\lambda = \nu \pm i\tau$ with $|\tau| \ge 100$ is not within the region allowed by Stirling's formula, we write $|\Gamma(\lambda)| = \frac{\pi}{|\Gamma(1-\lambda) \sin(\pi \lambda)|} \asymp \frac{1}{e^{\pi |\tau|} |\Gamma(1-\lambda)|}$, and apply Stirling's formula to $\Gamma(1-\lambda)$ instead.

Moreover, by \eqref{eqn:Huv}, 
\begin{align*}
&\left(\int_{b_3 - iT}^{ b_3 + iT} + \int_{b_3 + iT}^{-\infty + iT} + \int_{-\infty - iT}^{b_3 - iT}\right)\mathcal H(w, z) \left( \frac xy\right)^{-w} \> dw\\
&= \left(\int_{b_3 - iT}^{ b_3 + iT} + \int_{b_3 + iT}^{-\infty + iT} + \int_{-\infty - iT}^{b_3 - iT}\right)\left[\frac{\Gamma(w)\Gamma(z-w)}{\Gamma(z)}+ \frac{\Gamma(w) \Gamma(1-z)}{\Gamma(1+w-z)} + \frac{\Gamma(z-w) \Gamma(1-z)}{\Gamma(1-w)} \right] \left( \frac xy\right)^{-w} \> dw \\
&=  \left|1 + \frac xy \right|^{-z} + \left| 1 - \frac xy\right|^{-z},
\end{align*}upon summing the residues of the integrand within the region bounded by our contour of integration. 

This and \eqref{eqn:betaintegral} gives that
\begin{equation}
\left(\int_{b_3 - iT}^{ b_3 + iT} + \int_{b_3 + iT}^{-\infty + iT} + \int_{-\infty - iT}^{b_3 - iT}\right)\mathcal H(w, z) \left( \frac xy\right)^{-w} \> dw =  \frac{1}{2\pi i} \int_{(b_3)} \mathcal H(w, z) \left( \frac{x}{y} \right)^{-w}  dw
\end{equation}

Inserting the above and \eqref{eqn:horizonalbdd} into~\eqref{eq:Wint}, we obtain that, for $x < y$ and $0 < b_3 < c < 1,$ 
\es{\label{eqn:truncateWpm} \mathcal W(x, y; u) =  \ &\mathcal I (x, y, u; T)+  O\left( \frac{u^{-c} y^{-b_1 + b_3 - c} x^{-b_1 - b_3}}{ T^{1 - c}\left|\log\left( \frac xy\right) \right| } \right),}
where 
\est{\mathcal {I}(x, y, u ; T) = &\frac{1}{(2\pi i)^3} \int_{(b_1)} \int_{(c)} \int_{b_3 - iT/2}^{b_3 + iT/2} \int_{-\infty}^{\infty} u^{-z} \widetilde{\Psi}(1 + 4s + z) \mathcal H(w,z)  G\left(\frac 12 + s, t \right)\\
 &\hskip 1in \times \frac{y^{-s + it + w - z} x^{-s - it - w}}{\pi^{4s}}  \> dt \> dw \> dz \> \frac{ds}{s}}

Now we change the variable $s_1 = s + it + w$ and $s_2 = s - it - w + z$. Recall that $T \ge Q^{\varepsilon}.$ Let $\omega$ and $\xi$ be defined as in the statement of the lemma. We then have for $|c_1 - c_2| < c$
\est{ \mathcal W(x, y; u) = &\frac{1}{(2\pi i)^3}  \int_{(c)} \int_{c_1 - iT}^{c_1 + iT} \int_{c_2 - iT}^{c_2 + iT}  \widetilde{\W}_3(s_1, s_2; z) u^{-z} {y^{-s_2} x^{-s_1}}   \> ds_2 \> ds_1  \> dz \\
& \qquad \qquad +  O\left(Q^{-1000} + \frac{ u^{-c} y^{-c_2} x^{-c_1}}{ T^{1 - c}\left|\log\left( \frac xy\right) \right| }  \right)}
as desired. 
\end{proof}

The following lemma is of crucial importance when we remove the need for GRH --- in~\cite{CL8th, CIS} the decay in $|s_1+s_2|$ was not noticed or utilized.
\begin{lemma} \label{lem:MellinXY}
Given a positive real number $u,$ we define
$$ \Wt^\pm_2 (s_1, s_2; u) = \int_0^\infty \int_0^\infty \W^\pm(x, y ; u) x^{s_1}y^{s_2} \frac{dx}{x} \frac{dy}{y}.$$
Then the functions $\Wt^\pm_2(s_1, s_2 ; u)$ are analytic in the region $\tRe s_1, \tRe s_2 > 0$. We have the Mellin inversion formula
$$ \W^\pm(x, y ; u) = \frac{1}{(2\pi i)^2} \int_{(c_1)}\int_{(c_2)} \Wt^\pm_2 (s_1, s_2 ; u) x^{-s_1} y^{-s_2} \> d s_2 \>d s_1,$$
when $c_1$ and $c_2$ are positive. For any integers $k \geq 1, l \geq 0$ and any $s_1, s_2 \in \mathbb{C}$ with $0 < \tRe s_1, \tRe s_2 \leq 100$, one has
\begin{equation}
\label{eq:Wt2bound}
|\Wt^\pm_2(s_1, s_2; u)| \ll \frac{1}{\tRe s_1 \cdot \tRe s_2} \cdot \frac{(1 + u)^{k-1}}{\max\{|s_1|, |s_2|\}^k |s_1 + s_2|^l} \exp\left(-c' u^{-1/4}\right)
\end{equation}
for some constant $c' > 0$.
\end{lemma}

\begin{proof}
Apart from~\eqref{eq:Wt2bound} the claims go back to~\cite{CIS} (and are easy to prove). In the proof of~\eqref{eq:Wt2bound} we can assume by symmetry that $|s_1| \geq |s_2|$. The proof is based on partial integration, and for this we first study the derivatives of $\W^\pm(x, y ; u)$. 

To simplify the notation, we write in this proof $V_1, V_2, \dotsc$ for unspecified functions $V_j \colon \mathbb{R}_+^3 \to \mathbb{C}$ that satisfy~\eqref{eq:Vbound} and $W_1^\pm, W_2^\pm, \dotsc$ for unspecified functions $W_j^\pm \colon \mathbb{R}_+^3 \to \mathbb{C}$ that are of the form $W_j^\pm(x, y; u) = u|x\pm y| \Psi(u|x \pm y|) V_j(x, y; u |x\pm y|)$ for some such $V_j$.

By Lemma~\ref{lem:weightWV} and the chain rule
\[
\frac{d}{dx} V(x, y; u |x\pm y|) = \frac{1}{x} V_1(x, y; u |x \pm y|) + u \cdot \frac{1}{u|x\pm y|} V_2(x, y; u |x \pm y|),
\]
so that in the region $u |x \pm y| \in [1, 2]$, we have, for any $j \geq 0$
\[
\frac{d^j}{dx^j} V(x, y; u |x\pm y|) = \left(\frac{1}{x^j} + u^j\right) V_3(x, y; u |x \pm y|)
\]
and
\[
\frac{d^j}{dx^j} \W^\pm(x, y; u |x\pm y|) = \left(\frac{1}{x^j} + u^j\right) W_1(x, y; u |x \pm y|)
\]
Hence by partial integration $k$ times, we see that 
\[
|\Wt^\pm_2(s_1, s_2; u)| \ll \frac{1}{|s_1|^{k}} \left|\int_0^\infty \int_0^\infty (1+ux)^{k} W^\pm_1(x, y ; u) x^{s_1}y^{s_2} \frac{dx}{x} \frac{dy}{y}\right|,
\]
Now we substitute $x = w$ and $y = wz$, so that
\begin{equation}
\label{eq:Wt2bound1}
|\Wt^\pm_2(s_1, s_2; u)| \ll \frac{1}{|s_1|^{k}} \left|\int_0^\infty \int_0^\infty (1+uw)^{k} W^\pm_1(w, wz ; u) w^{s_1+s_2} z^{s_2} \frac{dw}{w} \frac{dz}{z}\right|.
\end{equation}

Next we perform partial integration with respect to $w$. For this note that, by Lemma~\ref{lem:weightWV} and the chain rule,
\[
\begin{split}
\frac{d}{dw} V(w, wz, uw|1\pm z|) &= \frac{1}{w} V_5(w, wz, uw|1\pm z|) + z \cdot \frac{1}{wz} V_6(w, wz, uw|1\pm z|) \\
& \qquad + u|1\pm z| \frac{1}{uw|1\pm z|}V_7(w, wz, uw|1\pm z|) \\
&= \frac{1}{w} V_8(w, wz, uw|1\pm z|).
\end{split}
\]
Note also that 
\[
\frac{d}{dw} (1+uw)^k = \frac{1}{w} (1+uw)^k \cdot k \frac{uw}{1+uw}.
\]
Using similar bounds for higher order derivatives and considering other terms as well, we see that 
\[
\frac{d^l}{dw^l} \left((1+uw)^{k} W^\pm_1(w, wz ; u)\right) = \frac{1}{w^l} (1+uw)^{k} W^\pm_2(w, wz ; u).
\]
Hence, by applying partial integration $l$ times to the right hand side of~\eqref{eq:Wt2bound1}, we obtain 
\[
\begin{split}
|\Wt^\pm_2(s_1, s_2; u)| &\ll \frac{1}{|s_1|^k |s_1+s_2|^l} \Bigl| \int_0^\infty \int_0^\infty w^{s_1+s_2} z^{s_2} (1+uw)^{k} W^\pm_3(w, wz ; u) \frac{dz}{z} \frac{dw}{w}\Bigr|
\end{split}
\]
Substituting back $w = x$ and $z = y/x$, we see that the above integral is
\begin{equation}
\label{eq:W2estint}
\int_0^\infty \int_0^\infty x^{s_1} y^{s_2} (1+ux)^{k} W^\pm_3(x, y; u) \frac{dy}{y} \frac{dx}{x}.
\end{equation}
Recall that $k \geq 1$ and $\tRe s_i > 0$ and that $W^\pm_3(x, y; u)$ is supported on $u|x \pm y| \in [1, 2]$ and satisfies $W^\pm_3(x, y; u) \ll \exp\left(-2c\max\{x, y\}^{1/4}\right)$ for some $c > 0$.

In the part of~\eqref{eq:W2estint} with $x > 10/u$, we have $y \asymp x$ and $y$ is restricted to an interval of length $1/u$ (depending on $x$). Hence the contribution of $x > 10/u$ to~\eqref{eq:W2estint} is bounded by
\[
\begin{split}
&u^{k-1} \int_{10/u}^\infty x^{\tRe s_1 + \tRe s_2} x^{k} \exp(-2c x^{1/4}) \frac{dx}{x^2} \\
& \ll u^{k-1} \frac{1}{\tRe s_1 + \tRe s_2 + k -1} \left(\frac{1}{u^{\tRe s_1 + \tRe s_2 + k -1}}  + 1\right) \exp(-2c u^{-1/4}) \\
& \ll u^{k-1} \frac{1}{\tRe s_1 \cdot \tRe s_2} \exp(-c u^{-1/4}).
\end{split}
\]
In the part of~\eqref{eq:W2estint} with $x \leq 10/u$, we have $y \leq 12/u$ and $\max\{x, y\} \geq 1/(2u)$ (since $|x \pm y| \in [1/u, 2/u]$), so the contribution of this part is bounded by
\[
\begin{split}
&\int_0^{10/u} \int_0^{12/u} \mathbf{1}_{\max\{x, y\} \geq 1/(2u)} x^{\tRe s_1} y^{\tRe s_2} \left|W^\pm_3(x, y; u)\right| \frac{dx}{x} \frac{dy}{y} \\
&\ll \frac{1}{\tRe s_1} \frac{1}{\tRe s_2} u^{-\tRe s_1 - \tRe s_2} \exp(-2c u^{-1/4}) \ll \frac{1}{\tRe s_1} \frac{1}{\tRe s_2} \exp(-c u^{-1/4}).
\end{split}
\]
\end{proof}

\subsection{Bounding the error term $\Eg$} \label{sec:errorEg}
The aim of this section is to prove the following.
\begin{lem} \label{lem:EG1bound} Let $\mathcal{EG}(\Psi, Q)$ be as in~\eqref{eqn:errortermG} with
\begin{equation}
\label{eq:ADQ_0def}
A = \exp((\log Q)^{\varepsilon_0/2}), \quad D = (\log Q)^{\Delta_0}, \quad  \text{and} \quad Q_0 = \exp((\log Q)^{\varepsilon_0}).
\end{equation}
Then
$$\mathcal{EG}(\Psi, Q) \ll_C \frac{Q^2}{(\log Q)^C}$$
for any $C \geq 1$.
\end{lem}

\begin{proof}
Following~\cite[Proof of Lemma 7.1]{CL8th} until~\cite[(25)]{CL8th}, we see that $\mathcal{EG}(\Psi, Q)$ is
\begin{align}
\label{eqn:offdiag}
&\ll Q (\log Q)^{16} \sum_{\substack{a \leq A \\ b, h > 0}} \sum_{\substack{\chi \ ({\rm mod} \ abh) \\ \chi \neq \chi_0}} \sum_{\substack{g \\ b|g, (a,g) = 1}} \sum_{\substack{d \leq D \\ (d, g) = 1}}  \frac{\tau^3(d)\tau^3(g)\tau_4(g) }{ad g\phi(abh)} \notag \\
& \cdot \int_{(\frac{100}{\log Q})}\int_{(\frac{100}{\log Q})} \left(\left|L^4(1/2+s_1, \chi)L^4(1/2+s_2, \overline{\chi})\right| + \tau_4^2 (g) \right) \left|\Wt_2^\pm\left(s_1, s_2; \frac{dQ}{gh \lotwo}\right)\right| ds_1 ds_2.
\end{align} 
Notice first that Lemma \ref{lem:MellinXY} implies that, for any $k \geq 1$,
\begin{align}
\label{eqn:offdiagdivsum}
\begin{split}
\sum_{\substack{g \\ b|g}} & \frac{\tau^3(g) \tau_4^3(g)}{g} \sum_{d \leq D}\frac{\tau^3(d)}{d}  \left|\widetilde{\W}^\pm_2\left(s_1, s_2, \frac{dQ}{gh\lotwo}\right)\right|  \\
&\ll  \frac{\left(1 + \frac{QD}{bh\lotwo}\right)^{k-1} (\log Q)^2}{\max\{|s_1|, |s_2|\}^k |s_1 + s_2|^3} (\log D)^8 \sum_{\substack{g \\ b|g}} \frac{\tau^3(g) \tau_4^3(g)}{g} \exp\left(-c \left( \frac{gh\lotwo}{QD}\right)^{1/4}\right) \\
&\ll   \frac{\left(1 + \frac{QD}{bh\lotwo}\right)^{k-1} (\log Q)^{O_k(1)}}{\max\{|s_1|+1, |s_2|+1\}^k (|s_1 + s_2| + 1)^3} \cdot \frac{\tau^3(b) \tau_4^3(b)}{b} \cdot \exp\left(-c \left( \frac{bh\lotwo}{QD}\right)^{1/4}\right)
\end{split}
\end{align}

Let us write $\ell = abh$ in~\eqref{eqn:offdiag} and split into dyadic blocks $\ell \sim L$ and $\max\{|s_1|+1, |s_2|+1\} \sim T$. Note that $\phi(abh) \gg abh / \log\log (abh)$ and $bh \geq L/A$. Using also the inequality $xy \leq x^2+y^2$, the contribution of such a dyadic block to~\eqref{eqn:offdiag} is \begin{align}
\label{eqn:offdiagdiadic}
\begin{split}
&\ll Q (\log Q)^{O(1)} \min_{k \in \{1, 4\}}  \frac{\left(1 + \frac{ADQ}{L \lotwo}\right)^{k-1} \log^\varepsilon (L)}{L T^k} \exp\left(-c \left( \frac{L\lotwo}{ADQ}\right)^{1/4}\right) \\
& \quad \cdot \sum_{\ell \sim L} \left(\sum_{\substack{abh = \ell \\ a \leq A \\ b, h > 0}} \frac{\tau^3(b) \tau_4^3(b)}{ab}\right) \sum_{\substack{\chi \ ( {\rm mod} \ \ell) \\ \chi \neq \chi_0}} \int_{-T}^T \left( 1 + \left|L^8(\tfrac{1}{2}+\tfrac{100}{\log Q}+it, \chi)\right|\right)  \>dt.
\end{split}
\end{align}
Here it suffices to make a rough estimate
\begin{equation}
\label{eq:tauellrough}
\sum_{\substack{abh = \ell \\ a \leq A \\ b, h > 0}} \frac{\tau^3(b) \tau_4^3(b)}{ab} \ll \log A \prod_{p \mid \ell} \left(1+\frac{2^{9}}{p}\right) \ll \log A \cdot \log L,
\end{equation}
so that, by Proposition \ref{prop:8momentIntandSumoverq}, the second line of~\eqref{eqn:offdiagdiadic} is at most
\begin{align}
\label{eqn:offdiag8m}
\begin{split}
& L^2 T^2 (\log (10 LQT))^{O(1)} .
\end{split}
\end{align}

By (\ref{eqn:offdiagdivsum}), (\ref{eqn:offdiagdiadic}), and (\ref{eqn:offdiag8m}) above, (\ref{eqn:offdiag}) is bounded by 
\begin{align*}
& Q(\log Q)^{O(1)} \sum_{\substack{L, T \geq 1 \\ L = 2^u, T = 2^v}} \min_{k \in \{1, 4\}} L \frac{\left(1 + \frac{ADQ}{L\lotwo}\right)^{k-1} \log^{O(1)} (LT)}{T^{k-2}} \exp\left(-c \left( \frac{L\lotwo}{ADQ}\right)^{1/4}\right)
\end{align*}
Taking $k = 1$ when $T \leq 1 + \frac{ADQ}{L \lotwo}$ and $k =4$ otherwise, we see that this is 
\begin{align*}
&\ll Q \sum_{\substack{L \geq 1 \\ L = 2^u}} L \left(1 + \frac{ADQ}{L\lotwo}\right) \log^{O(1)} (LQ) \exp\left(-c \left( \frac{L\lotwo}{ADQ}\right)^{1/4}\right) \\
&\ll \frac{ADQ^2(\log Q)^{O(1)}}{Q_0^2} \ll_C \frac{Q^2}{(\log Q)^C}
\end{align*}
for any $C \geq 1$. Hence the claim follows.
\end{proof}

\begin{rem}
\label{rem:aleqA}
In~\cite{CL8th} the conditional bound~\eqref{eq:CL8thupperGRH} was used instead of Proposition~\ref{prop:8momentIntandSumoverq}. Since~\eqref{eq:CL8thupperGRH} does not require averaging over the modulus, in~\cite{CL8th} it was possible to utilize averaging over $\ell$ when bounding the left hand side of~\eqref{eq:tauellrough}. Due to this, the restriction $a \leq A$ (which is in place to make~\eqref{eq:tauellrough} hold) was not needed in the definition of $\mathcal{G}^\pm(\Psi, Q)$ in~\cite{CL8th}. In our case, we needed to separate the terms with $a > A$ and treat them by a variant of the treatment of the case $d > D$ in~\cite{CL8th}.
\end{rem}

\section{Evaluating $\mathcal{MS}(\Psi, Q) + \mathcal{MG}(\Psi, Q)$} 
\label{sec:MsplusMg}
We recall that $\mathcal{MS}(\Psi, Q)$ and $\mathcal{MG}(\Psi, Q)$ are defined in
(\ref{eq:MSdef}) and  (\ref{eqn:maintermG}), respectively. We evaluate $\mathcal{MS}(\Psi, Q) + \Mg$ following~\cite[Section 8]{CL8th} but the details are somewhat different since we have the restriction $a \leq A$ in our sums. 

Similarly to~\cite[Section 8]{CL8th} we use the Mellin transform of $\Wt_1^\pm$ (from Lemma \ref{lem:MellinU}) to write $\mathcal{MG}(\Psi, Q)$ in terms of a contour integral with $\tRe z = -\varepsilon < 0$ and shift the contour to $\R(z) = \varepsilon > 0$. Here we pick up poles at $z = 0$ whose residue essentially cancels with $\mathcal{MS}(\Psi, Q)$. This process is recorded in the following lemma.

\begin{lem} \label{lem:MSplusMG} Let $C \geq 1$. Let $\mathcal{MS}(\Psi, Q)$ and $\mathcal{MG}(\Psi, Q)$ be as in (\ref{eq:MSdef}) and (\ref{eqn:maintermG}) with parameters as in~\eqref{eq:ADQ_0def}. Once $\Delta_0$ is large enough in terms of $C$, one has
\begin{align} \label{eqn:maintermsumMSandMG} 
\begin{split}
\mathcal{MS}(\Phi, Q) + \mathcal{MG}(\Phi, Q) = \frac{Q}{2} &\sum_{\substack{m, n = 1 \\ m \neq n}}^{\infty} \frac{\tau_4 (m) \tau_4 (n)}{\sqrt{mn}}  \frac{1}{2\pi i} \int_{(\varepsilon)} \Wt^\pm_1\left(\frac{m \lotwo}{Q^2}, \frac{n \lotwo}{Q^2} ; z\right) \\
&\cdot \frac{\zeta(1 - z) \F(-z, g, MN)}{\zeta(1 + z) \phi(gMN, 1+ z)} \left(\frac{Q}{g \lotwo}\right)^{-z} \> dz + O\left( \frac{Q^2}{(\log Q)^C}\right),
\end{split}
\end{align}
where 
$$ \phi(r, s)= \prod_{p | r} \left( 1 - \frac{1}{p^s}\right),$$
and
\[
\F(s, g, MN) = \phi(MN, s + 1) \prod_{p \nmid gmN} \left(1 - \frac{1}{p(p-1)} + \frac{1}{p^{1 + s}(p-1)}\right) \prod_{\substack{p|g \\ p \nmid MN}} \left( 1 - \frac{1}{p^{1 + s}} - \frac{1}{p-1} \left(1 - \frac{1}{p^s}\right)\right).
\] 
\end{lem} 

\begin{proof}
Write
\[
S(a, b, d, g, M, N) := \sum_{\substack{h > 0 \\ (h, MN) = 1}}  \frac{1}{\phi(abh)} \W^\pm \left(\frac{gM \lotwo}{Q^2}, \frac{gN \lotwo}{Q^2} ; \frac{dQ}{gh \lotwo}\right),
\]
so that
\[
\mathcal{MG}(\Psi, Q) = \frac{Q}{2} \sum_{\substack{m, n = 1\\ m \neq n}}^{\infty} \frac{\tau_4(m) \tau_4(n)}{\sqrt{mn}}\sum_{\substack{d \leq D \\ (d, gMN) = 1}}  \sum_{\substack{a \leq A \\ (a, gMN)=1}} \sum_{\substack{b|g \\ (b, MN) = 1}} \frac{\mu(a) \mu(b) \mu(d) }{ad} S(a, b, d, g, M, N).
\]
where as usual $g := (m,n)$, $M := m / g$ and $N := n / g$.

Using the Mellin transform of $\Wt_1^\pm$ given in Lemma \ref{lem:MellinU} with $c = - \varepsilon < 0$, we obtain that
\begin{equation} \label{eqn:maintermGsumoverh}
S(a, b, d, g, M, N) = \sum_{\substack{h = 1 \\(h, MN) = 1}}^{\infty}\frac{1}{\phi(abh)} \frac{1}{2\pi i} \int_{(-\varepsilon)} \Wt_1^\pm \left(\frac{gM \lotwo}{Q^2}, \frac{gN \lotwo}{Q^2} ; z\right) \left(\frac{dQ}{gh \lotwo}\right)^{-z} \> dz.
\end{equation}
We can interchange the sum and the integral since the sum over $h$ is absolutely convergent for $\tRe (z)< 0$. Writing out the Euler product, we obtain that, for $(ab, MN) = 1$ and $\tRe (s) > 0$, one has
\begin{equation}
\label{eq:EulerProduct1perphi}
\begin{split}
\sum_{\substack{h = 1 \\(h, MN) = 1}}^{\infty} \frac{1}{\phi(abh) h^s} 
& = \frac{1}{\phi(ab)} \zeta(s+1) \prod_{p \mid MN} \left(1-\frac{1}{p^{s+1}}\right)\prod_{p \nmid abMN} \left(1+\frac{1}{p^{s}} \cdot \frac{1}{p(p-1)}\right).
\end{split}
\end{equation}
Therefore
\[
\begin{split}
S(a, b, d, g, M, N) &= \frac{1}{2\pi i} \int_{(-\varepsilon)} \Wt^\pm_1\left(\frac{gM \lotwo}{Q^2}, \frac{gN \lotwo}{Q^2} ; z\right) \zeta(1 - z) \left(\frac{dQ}{g \lotwo}\right)^{-z} \\
& \qquad \cdot \frac{1}{\phi(ab)} \prod_{p \mid MN} \left(1-\frac{1}{p^{1-z}}\right)\prod_{p \nmid abMN} \left(1+\frac{p^z}{p(p-1)}\right) \> dz .
\end{split}
\]
Next we move the integration line to $\tRe z = \varepsilon$. We encounter a pole at $z = 0$, leading to a main term $-({\rm Residue \ at} \ z = 0)$ that equals  
\begin{align*} 
&\Wt^\pm_1\left(\frac{gM \lotwo}{Q^2}, \frac{gN \lotwo}{Q^2} ; 0\right) \frac{1}{\phi(ab)} \frac{\phi(MN)}{MN} \prod_{p \nmid abMN} \left(1+\frac{1}{p(p-1)}\right).
\end{align*}
Here, by first substituting $u' = u|x\pm y|$ in the definition of $\Wt^\pm_1$ and then using~\eqref{eqn:Vc},
\[
\Wt^\pm_1\left(\frac{gM \lotwo}{Q^2}, \frac{gN \lotwo}{Q^2} ; 0\right) = \int_0^\infty \Psi(u) V\left(\frac{gM \lotwo}{Q^2}, \frac{gN \lotwo}{Q^2}, u\right) du = \int_0^\infty \Psi(u) V\left(gM, gN, \frac{uQ}{Q_0}\right) du.
\]
Therefore, writing $\mathcal{MG}_1(\Psi, Q)$ for the contribution of the residue term at $z = 0$ we obtain
\begin{equation}
\label{eq:MG=MG1+MG2}
\mathcal{MG}(\Psi, Q) = \mathcal{MG}_1(\Psi, Q) + \mathcal{MG}_2(\Psi, Q),
\end{equation}
where
\begin{equation} \label{eqn:residue0ofMg}
\begin{split}
\mathcal{MG}_1(\Psi, Q) &:= Q \sum_{\substack{m, n = 1  \\ m \neq n}} \frac{\tau_4 (m) \tau_4 (n)}{\sqrt{mn}} \frac{\phi(MN)}{MN}  \sum_{\substack{d \leq D \\ (d, mn) = 1}}\frac{\mu(d)}{d} \sum_{\substack{a \leq A \\ (a, mn) = 1}} \sum_{\substack{b \mid g \\ (b, MN) = 1}} \frac{\mu(a)\mu(b)}{a\phi(ab)} \\
& \qquad \cdot  \prod_{p \nmid abMN} \left(1+\frac{1}{p(p-1)}\right) \int_0^\infty  \Psi(u) V\left(gM, gN, \frac{uQ}{Q_0}\right) du,
\end{split}
\end{equation}
where we have changed the factor of $Q/2$ to $Q$ due to summing over $\pm$.  Moreover,
\begin{equation}
\label{eq:Mg2def}
\begin{split}
&\mathcal{MG}_2(\Psi, Q) := \frac{Q}{2} \sum_{\substack{m, n = 1\\ m \neq n}}^{\infty} \frac{\tau_4(m) \tau_4(n)}{\sqrt{mn}}\frac{1}{2\pi i} \int_{(\varepsilon)}  \sum_{\substack{d \leq D \\ (d, gMN) = 1}}  \sum_{\substack{a \leq A \\ (a, mn)=1}} \sum_{\substack{b|g \\ (b, MN) = 1}} \frac{\mu(a) \mu(b) \mu(d) }{a\phi(ab)d^{1+z}} \\
& \cdot \Wt^\pm_1\left(\frac{gM \lotwo}{Q^2}, \frac{gN \lotwo}{Q^2} ; z\right) \zeta(1 - z) \left(\frac{Q}{g \lotwo}\right)^{-z} \prod_{p \mid MN} \left(1-\frac{1}{p^{1-z}}\right)\prod_{p \nmid abMN} \left(1+\frac{p^z}{p(p-1)}\right) \> dz.
\end{split}
\end{equation}

We will first show that
\begin{equation}
\label{eq:MG1+MS}
\mathcal{MG}_1(\Psi, Q) + \mathcal{MS}(\Psi, Q) = O\left(\frac{DQ^2(\log Q)^{O(1)}}{Q_0}\right).
\end{equation}
By definition~\eqref{eq:MSdef}, 
\[
\begin{split}
\mathcal {MS}(\Psi, Q) &=- \sum_{\substack{m, n = 1  \\ m \neq n}} \frac{\tau_4 (m) \tau_4 (n)}{\sqrt{mn}} \sum_{\substack{ a \leq A, d \leq D \\  (ad, mn) = 1}} \sum_{\substack{l > 0 \\ (l, MN) = 1}} \mathbf{1}_{(l, g) = 1} \Psi \left( \frac{adl}{Q}\right) \frac{\mu(a) \mu(d) l}{\phi(al)}   V\left(m, n ; \frac{adl}{\lon} \right).
\end{split}
\]
Writing $\mathbf{1}_{(l, g)=1} = \sum_{\substack{l = bk, b \mid g}} \mu(b)$, we see that
\[
\begin{split}
\mathcal {MS}(\Psi, Q) &=  - \sum_{\substack{m, n = 1  \\ m \neq n}} \frac{\tau_4 (m) \tau_4 (n)}{\sqrt{mn}} \sum_{\substack{ a \leq A, d \leq D \\ (ad, mn) = 1}} \sum_{\substack{b \mid g, k > 0 \\ (bk, MN) = 1}} \Psi \left( \frac{abdk}{Q}\right) \frac{\mu(a) \mu(b) \mu(d) bk}{\phi(abk)}   V\left(m, n ; \frac{abdk}{\lon} \right).
\end{split}
\]
Writing $\mathcal{U}(m, n, u) = V(m, n, u/\lon) \Psi(u/Q)$ and using Mellin inversion, we see that
\[
\mathcal {MS}(\Psi, Q) = - \sum_{\substack{m, n = 1  \\ m \neq n}} \frac{\tau_4 (m) \tau_4 (n)}{\sqrt{mn}} \sum_{\substack{ a \leq A, d \leq D \\ (ad, mn) = 1}} \sum_{\substack{b \mid g \\ (b, MN) = 1}} \sum_{\substack{k > 0 \\ (k, MN) = 1}}  \frac{\mu(a) \mu(b) \mu(d) bk}{\phi(abk)} \frac{1}{2\pi i} \int_{(1+\varepsilon)} \frac{\widetilde{\mathcal{U}}(m, n, z)}{(abdk)^{z}} dz,
\]
where
\[
\widetilde{\mathcal{U}}(m, n, z) = \int_0^\infty \mathcal{U}(m, n, u) u^z \frac{du}{u}.
\]
Noting that the sum over $k$ above is absolutely convergent for $\tRe z > 1$, we can interchange the order of summation and integration. Recalling~\eqref{eq:EulerProduct1perphi} we see that 
\[
\begin{split}
\mathcal {MS}(\Psi, Q) &= - \sum_{\substack{m, n = 1  \\ m \neq n}} \frac{\tau_4 (m) \tau_4 (n)}{\sqrt{mn}} \sum_{\substack{b \mid g \\ (b, MN) = 1}} \frac{1}{2\pi i} \int_{(1+\varepsilon)} \sum_{\substack{ a \leq A, d \leq D \\ (ad, mn) = 1}}  \frac{\mu(a) \mu(b) \mu(d) b}{\phi(ab)(abd)^z} \zeta(z)\\
& \qquad \qquad \cdot \prod_{p \mid MN} \left(1-\frac{1}{p^z}\right) \prod_{p \nmid abMN}\left(1+\frac{1}{p^{z-1}} \cdot \frac{1}{p(p-1)}\right) \widetilde{\mathcal{U}}(m, n, z) dz.
\end{split}
\]
Next we move the integration to the line $\tRe z = 1/\log Q$. Since
\[
\widetilde{\mathcal{U}}(m, n, 1) = Q \int_0^\infty \Psi(u) V\left(m, n; \frac{uQ}{\lon}\right) du,
\]
the residue from the pole at $z = 1$ equals $-\mathcal{MG}_1(\Psi, Q)$ and the remaining integral can be included in the error term in~\eqref{eq:MG1+MS} since $\widetilde{\mathcal{U}}(m, n; z)$ decays rapidly when $m$ or $n$ is $\gg Q^2/Q_0$ or $|\tIm z|$ grows.

Therefore~\eqref{eq:MG1+MS} holds and the remaining main term of $\mathcal{MG}(\Psi, Q) + \mathcal{MS}(\Psi, Q)$ is $\mathcal{MG}_2(\Psi, Q)$ defined by~\eqref{eq:Mg2def}. We shift the contour in~\eqref{eq:Mg2def} to ${\rm Re}(z) = 1 - \frac{1}{\log Q}.$ The sums over $a$ and $d$ can be extended to all positive integers with an error $\ll Q^2/(\log Q)^C$ using Lemma~\ref{lem:MellinU}. For the $d$-sum, this was done in~\cite[Proof of Lemma 8.1]{CL8th} and one can argue similarly for the $a$-sum. Hence, apart from an acceptable error, $\mathcal{MG}_2(\Psi, Q)$ equals 
\[
\begin{split}
&\frac{Q}{2} \sum_{\substack{m, n = 1\\ m \neq n}}^{\infty} \frac{\tau_4(m) \tau_4(n)}{\sqrt{mn}}\frac{1}{2\pi i} \int_{(1-1/\log Q)} \Wt^\pm_1\left(\frac{gM \lotwo}{Q^2}, \frac{gN \lotwo}{Q^2} ; z\right) \zeta(1 - z) \left(\frac{Q}{g \lotwo}\right)^{-z} \\
&\quad \cdot  \sum_{\substack{d > 0 \\ (d, gMN) = 1}}  \sum_{\substack{a >0 \\ (a, mn)=1}} \sum_{\substack{b|g \\ (b, MN) = 1}} \frac{\mu(a) \mu(b) \mu(d) }{a\phi(ab)d^{1+z}}  \prod_{p \mid MN} \left(1-\frac{1}{p^{1-z}}\right)\prod_{p \nmid abMN} \left(1+\frac{p^z}{p(p-1)}\right) \> dz.
\end{split}
\]

Writing $r = ab$, the second line equals
\[
\begin{split}
& = \sum_{\substack{d > 0 \\ (d, gMN) = 1}}  \sum_{\substack{r > 0 \\ (r, MN) = 1}} \frac{\mu(d) \mu(r) (r, g)}{r \phi(r)d^{1+z}} \prod_{p \mid MN} \left(1-\frac{1}{p^{1-z}}\right)\prod_{p \nmid rMN} \left(1+\frac{p^z}{p(p-1)}\right).
\end{split}
\]
Careful calculation with the Euler products reveals that this equals
\[
\frac{\mathcal{F}(-z, g, MN)}{\zeta(1+z) \phi(gMN, 1+z)}
\]
and thus
\[
\begin{split}
\mathcal{MG}_2(\Psi, Q) &=\frac{Q}{2} \sum_{\substack{m, n = 1\\ m \neq n}}^{\infty} \frac{\tau_4(m) \tau_4(n)}{\sqrt{mn}} \frac{1}{2\pi i} \int_{(1-1/\log Q)} \frac{\zeta(1 - z)\mathcal{F}(-z, g, MN)}{\zeta(1+z) \phi(gMN, 1+z)} \\
& \qquad \qquad \qquad \cdot \Wt^\pm_1\left(\frac{gM \lotwo}{Q^2}, \frac{gN \lotwo}{Q^2} ; z\right)  \left(\frac{Q}{g \lotwo}\right)^{-z} \> dz.
\end{split}
\]
Now the Lemma follows from moving the contour back to $\tRe z = \varepsilon$ along with \eqref{eq:MG=MG1+MG2} and~\eqref{eq:MG1+MS}.
\end{proof}

Next we will prove the following proposition which evaluates the main term in  Lemma \ref{lem:MSplusMG}. In~\cite{CL8th} this proposition was conditional on the Lindel\"{o}f hypothesis, but here we prove it unconditionally.

\begin{prop} \label{prop:mainMG}  Let $\mathcal{MS}(\Psi, Q)$ and $\mathcal{MG}(\Psi, Q)$ be as in (\ref{eq:MSdef}) and (\ref{eqn:maintermG}) with parameters as in~\eqref{eq:ADQ_0def}. Once $\Delta_0$ is large enough, one has
$$  \mathcal{MS}(\Psi, Q) + \Mg = \frac{-53524}{16!}Q^2(\log Q)^{16}\frac{\psit(2)}{2} \frac{\mathcal{K}\left(\h, \h; 1 \right)}{\zeta(2)}\int_{-\infty}^{\infty} G\left(\h, t \right) \> dt + O(Q^2 (\log Q)^{15 }).$$
\end{prop}
with $\mathcal{K}(s_1, s_2; z)$ defined in \cite[Equation (38)]{CL8th}.
\begin{proof}
First we apply \eqref{eqn:truncate2} from Lemma \ref{lem:MellinXYU} to~\eqref{eqn:maintermsumMSandMG} to obtain that 

\es{
\label{eq:MTwithW3}
&\mathcal{MS}(\Psi, Q) + \Mg = \frac{Q}{2} \sum_{\substack{m, n = 1 \\ m \neq n}}^{\infty} \frac{\tau_4 (m) \tau_4 (n)}{\sqrt{mn}} \frac{1}{(2\pi i)^3} \int_{(\varepsilon)}  \int_{\frac 12 + \varepsilon - iT }^{\frac 12 + \varepsilon + iT} \int_{\frac 12 + \varepsilon - iT}^{\frac 12 + \varepsilon + iT}\Wt_3(s_1, s_2 ; z)  \\
&\quad \cdot \frac{\zeta(1 - z) \F(-z, g, MN)}{\zeta(1 + z) \phi(gMN, 1+ z)} \left(\frac{Q}{g \lotwo}\right)^{-z} \left(\frac{Q^2}{m Q_0^2}\right)^{s_1} \left(\frac{Q^2}{n Q_0^2}\right)^{s_2} \> ds_2 ds_1 dz \\
& \quad + O\left(\frac{Q^{3 + 3\varepsilon}}{T^{1 - \varepsilon}} \sum_{\substack{m, n = 1 \\ m \neq n}}^{\infty} \frac{\tau_4(m) \tau_4(n)}{m^{1 + \varepsilon} n^{1 + \varepsilon}} \frac{ 1}{ \left|\log\left( \frac mn\right) \right| } \right) + O(Q^2).
}
We choose the height $T := Q^{5/4}.$ Let us first consider the error term. We divide the sum over $m$ and $n$ into two ranges. The first case is when $|m - n| > \frac 14 n$. In this case, $\left|\log \left(\frac mn\right) \right| \gg 1$ and thus the contribution to the error term in~\eqref{eq:MTwithW3} is
\est{\ll Q^{\frac 74 + 5\varepsilon} \sum_{m, n = 1}^\infty \frac{\tau_4(m) \tau_4(n)}{(mn)^{1 + \varepsilon}} \ll Q^{\frac 74 + 5\varepsilon}. }
For the other range $|m - n| \leq \frac 14 n$ we have 
$$ \left| \log\left( \frac mn\right)\right| \gg \frac{|m - n|}{n}.$$
Thus the contribution from this range to the error term in~\eqref{eq:MTwithW3} is
\est{ \ll Q^{\frac 74 + 5\varepsilon}  \sum_{n = 1}^{\infty} \frac{\tau_4(n)}{n^{1 + \varepsilon}} \sum_{\substack{m \neq n \\ \frac{3}{4}n \leq m \leq \frac{5}{4} n}} \frac{\tau_4(m)}{m^{1 + \varepsilon}} \frac{n}{|m - n|} \ll  Q^{\frac 74 + 5\varepsilon}\sum_{n = 1}^{\infty} \frac{\tau_4(n)}{n^{1 + \varepsilon}} \sum_{j = 1}^{ n} \frac{1}{j} \ll Q^{\frac 74 + 5\varepsilon}. }

Before we move the contour integral, we reinsert the terms $m = n$ to the main term of~\eqref{eq:MTwithW3}. The contribution of this addition is
\est{ &\frac{Q}{2} \sum_{\substack{n = 1 }}^{\infty} \frac{\tau_4^2 (n) }{n} \frac{1}{(2\pi i)^3} \int_{(\varepsilon)}  \int_{\frac 12 + \varepsilon - iT }^{\frac 12 + \varepsilon + iT} \int_{\frac 12 + \varepsilon - iT}^{\frac 12 + \varepsilon + iT}\Wt_3(s_1, s_2 ; z)  \\
&\quad \cdot \frac{\zeta(1 - z) \F(-z, n, 1)}{\zeta(1 + z) \phi(n, 1+ z)} \left(\frac{Q}{n \lotwo}\right)^{-z} \left(\frac{Q^2}{n Q_0^2}\right)^{s_1 + s_2}  \> ds_2 ds_1 dz. }
Let us now show that this contribution is acceptable. We can move the contour integral over $s_1$ and $s_2$ to $\tRe(s_i) = \varepsilon$. We encounter no poles, and the sum over $n$ is absolutely convergent.  The resulting vertical integral is bounded by $Q^{1+3\varepsilon}$, and by \eqref{eqn:boundWt3}, the contribution from horizontal integrals is bounded by 
$$ \frac{Q^{3 + 4\varepsilon}}{T^{1 - \varepsilon}} \ll Q^{\frac 74 + 6\varepsilon}.$$
Hence, apart from an acceptable error,~\eqref{eq:MTwithW3} equals
\begin{align}
\label{eq:Ms+MgMellin3}
\frac{Q}{2} \frac{1}{(2\pi i)^3} \int_{(\varepsilon)}  \int_{1/2 + \varepsilon-iT}^{1/2+ \varepsilon+iT} \int_{1/2 + \varepsilon-iT}^{1/2+ \varepsilon+iT}& \Wt_3(s_1, s_2 ; z) \frac{\zeta(1 - z)}{\zeta(1 + z)} \frac{Q^{2s_1 + 2s_2 - z}}{Q_0^{2(s_1 + s_2 - z)}}  \mathcal{J}(s_1, s_2; z) \>ds_2 \>ds_1 \> dz, 
\end{align}
where 
\begin{align*}
\mathcal{J}(s_1, s_2; z) = \sum_{\substack{m, n = 1}}^{\infty} \frac{\tau_4 (m) \tau_4 (n)}{m^{1/2 + s_1} n^{1/2 + s_2}} \frac{g^z \F(-z, g, MN)}{\phi(gMN, 1+ z)}.
\end{align*}
As in~\cite[Proof of Proposition 9.1]{CL8th}, writing out the Euler product, one can see that
\begin{align*}
\mathcal{J}(s_1, s_2; z) = \zeta(2-z) \frac{\zeta^4\left(\frac{1}{2} + s_1\right)\zeta^4\left(\frac{1}{2} + s_2\right)}{\zeta^4\left(\frac{3}{2} + s_1 - z\right)\zeta^4\left(\frac{3}{2} + s_2 - z\right)} \zeta^{16}(1 + s_1 + s_2 - z) \mathcal{K} (s_1, s_2; z),
\end{align*}
where $\mathcal{K} (s_1, s_2; z) = \prod_p \mathcal{K}_p(s_1, s_2; z)$ is absolutely convergent when
\[
\tRe (s_1) > 0, \tRe (s_2) > 0, \tRe (s_1 + s_2) > \tRe(z) - 1/2
\]
and
\[
\tRe(z) \in (0, 3/2), \text{ and } \tRe z < 1 + \tRe s_i. 
\]


Now we move the lines of integration in~\eqref{eq:Ms+MgMellin3} to  $\tRe(s_1) = \tRe(s_2) = 2\varepsilon$, encountering poles of order four at $s_1 = 1/2$ and $s_2 = 1/2.$ By the Weyl bound (see e.g.~\cite[Theorem 5.12]{Ti}) and the Phragmen-Lindel\"of principle one has, for $\sigma \geq 0$, 
\[
|\zeta(1/2+\sigma+it)| \ll T^{\max\{0, 1/6-\sigma/3\} + \varepsilon},
\]
so that, using also~\eqref{eqn:boundWt3}, the horizontal integrals contribute
\begin{align*}
& Q^{1 + \varepsilon} \max_{2\varepsilon \leq \sigma \leq 1/2 + \varepsilon} \frac{T^{8\max\{0, 1/6-\sigma/3\}+8\varepsilon} Q^{4\sigma}}{T^{1-\varepsilon}} \\
&= Q^{1 + \varepsilon} \frac{Q^{4(1/2 + \varepsilon)}T^{9\varepsilon}}{T} + Q^{1 + \varepsilon} T^{8\varepsilon} \max_{2\varepsilon \leq \sigma \leq 1/2} T^{1/3} \left(\frac{Q}{T^{2/3}}\right)^{4\sigma} \\
&= O(Q^{7/4+ 20\varepsilon}).
\end{align*}

Furthermore, using again~\eqref{eqn:boundWt3}, the integrals over $\tRe s_i = 2\varepsilon$ contribute to~\eqref{eq:Ms+MgMellin3}
\[
\begin{split}
&\ll \max_{1 \leq T' \leq T} Q^{1+ \varepsilon} \frac{1}{T'^{1-\varepsilon}}\int_{-T'}^{T'} |\zeta(1/2+2/\log Q+it)|^8 dt. 
\end{split}
\]
By Lemma~\ref{le:zeta8th} this is 
\[
\ll Q^{1+\varepsilon} \max_{1 \leq T' \leq T} \frac{1}{T'^{1-\varepsilon}} T'^{3/2+\varepsilon} \ll Q^{1+4\varepsilon} T^{1/2} \ll Q^{7/4}.
\]
Finally by the definition of $\widetilde{\mathcal W}_3(s_1, s_2; z)$ in \eqref{eqn:mellinthree}, the main contribution of the residues is
\begin{align*} 
&\frac{Q}{2} \frac{1}{2\pi i} \int_{(\epsilon)} \frac{\psit(3 - z)}{ \pi^{2 - 2z}}\left( \int_{-\infty}^{\infty} \Hc \left(\frac{z}{2} - it, z\right) G\left(\h + \frac{1 - z}{2}, t \right) \> dt\right) \frac{\zeta(1 - z)}{\zeta(1 + z)} {\zeta(2-z) \mathcal{K}\left(\frac{1}{2}, \frac{1}{2}; z\right)} \\
& \times  \left({\rm Res}_{s_1 = s_2 = 1/2} \frac{\zeta^4\left(\frac{1}{2} + s_1\right)\zeta^4\left(\frac{1}{2} + s_2\right) \zeta^{16}(1 + s_1 + s_2 - z)}{(s_1 + s_2 -z)\zeta^4\left(\tfrac 32 + s_1 - z\right)\zeta^4\left( \tfrac 32 + s_2 - z \right)} \frac{Q^{2s_1 + 2s_2 - z}}{Q_0^{2(s_1 + s_2 - z)}}\right) \> dz. \nonumber
\end{align*}
Then we follow the residue calculation of Proposition 9.1 in \, \cite{CL8th} and obtain the main term.

\end{proof}

Now we have handled unconditionally all the terms involved and Theorem \ref{thm:eightmoment} follows as in~\cite[Section 10]{CL8th}.

\section{Acknowledgements}
The first and second authors acknowledge support from a Simons Travel Grant for Mathematicians. The first author is also supported by NSF grant DMS-2101806. The third author was supported by Academy of Finland grant no. 285894. The fourth author acknowledges support of NSF grant DMS-1902063.

This work was initiated while the first, third and fourth authors were in residence at MSRI in Spring 2017, which was supported by NSF grant DMS-1440140.

\bibliography{EightMoment}
\bibliographystyle{plain}

\end{document}